\documentclass{article}

\usepackage{lineno,hyperref}

\usepackage{floatrow}
\usepackage{graphicx,subcaption}
\usepackage{amsmath, amsfonts, amsthm, amssymb, bm}
\usepackage{color}
\usepackage{tabularx,hhline}
\usepackage{multirow}
\usepackage{booktabs}
\usepackage{fullpage}[2cm]
\usepackage{algorithm,setspace}
\usepackage{algorithmic}
\usepackage{enumerate}
\usepackage{paralist}
\usepackage{cancel}

\usepackage{xspace}
\usepackage{listings}
\usepackage[normalem]{ulem}
\usepackage{mathtools}
\usepackage{bbm}
\usepackage{tcolorbox}
\usepackage[toc,page]{appendix}
\usepackage{floatrow}

\usepackage[margin = 1.5cm, font = small, labelfont = bf, labelsep = period]{caption}
\usepackage[font = small, labelfont = bf, labelsep = period]{caption}
\newfloatcommand{capbtabbox}{table}[][\FBwidth]

\newcommand{\st}{\textnormal{s.t.}}

\newcommand{\PP}{\mathbb{P}}

\newcommand{\EE}{\mathbb{E}}
\newcommand{\RR}{\mathbb{R}}

\newcommand{\X}{\mathcal{X}}

\newcommand{\tr}{\textup{tr}}

\newtheorem{theorem}{Theorem}
\newtheorem{proposition}{Proposition}
\newtheorem{lemma}{Lemma}

\newtheorem{defi}{Definition}
\newtheorem{assume}{Assumption}

\begin{document}

%
%
%
%

\font\myfont=cmr12 at 15pt
\title{{\myfont A Decision Rule Approach for Two-Stage Data-Driven Distributionally Robust Optimization Problems with Random Recourse}}
\author{ Xiangyi Fan and Grani A. Hanasusanto\\ Graduate Program in Operations Research and Industrial Engineering\\ The University of Texas at Austin, USA}
\date{ }
\maketitle

\begin{abstract} 
We study two-stage stochastic optimization problems with random recourse, where the adaptive decisions are multiplied with the uncertain parameters in both the objective function and the constraints. To mitigate the computational intractability of infinite-dimensional optimization, we propose a scalable approximation scheme via piecewise linear and piecewise quadratic decision rules. We then develop a data-driven distributionally robust framework with two layers of robustness to address distributionally uncertainty. The emerging optimization problem can be reformulated as an exact copositive program, which admits tractable approximations in semidefinite programming. We design a decomposition algorithm where smaller-size semidefinite programs can be solved in parallel, which further reduces the runtime. Lastly, we establish the performance guarantees of the proposed scheme and demonstrate its effectiveness through numerical examples.
\end{abstract}

\section{Introduction}



Two-stage  optimization  under uncertainty models decision-making situations in which recourse actions can be taken once the realizations of the uncertain parameters are revealed. The setting is prevalent across many problems in operations management, transportation, and finance. 
The classical two-stage stochastic programs  assume that the uncertain parameters are random with complete knowledge on the joint probability distribution, which is rarely available in practice. Optimizing in view of the discrete empirical distribution based on the historical samples often yields inferior solutions that perform poorly in practice. To mitigate these overfitting effects, recent interest has grown in using the distributionally robust optimization (DRO) methodology~\cite{ DRO-solution-2, two-stage-sampling-robust, DRO-solution-3, DRO-tractable-1, Wass-SDP, DRO-solution-1}. In DRO, we construct an \emph{ambiguity set} of different plausible distributions that are consistent with the available information. Optimal decisions are then made in view of the worst-case probability distribution taken from within the ambiguity set. 
Hence, the DRO model yields  decisions that safely anticipate adverse outcomes and exhibit superior performance in out-of-sample tests. However, despite this promising result, the two-stage DRO problems are generically intractable to solve because they optimize over functions describing the recourse policies \cite{DRO-NPhard}.


Current solution schemes for two-stage DRO problems are focused on problems with fixed recourse, i.e., when the adaptive decisions are \emph{not} multiplied with the random parameters in the second stage. However, many optimization problems under uncertainty in finance \cite{Random-Recourse-Robust-1, Random-Recourse-Stochastic-2, Random-Recourse-Stochastic-4}, energy systems \cite{Random-Recourse-Stochastic-3}, and inventory control \cite{Random-Recourse-Stochastic-1}, etc.,  have \emph{random} (or non-fixed) recourse. In this paper, we consider the popular conservative approximation scheme via decision rules, which restricts the recourse variables to simple functions. For example, one can employ the affine, segregated affine, piecewise affine, or polynomial functions of the uncertain parameters, respectively denoted as linear decision rules (LDR) \cite{LDR1, LDR2, Linear-decision-rule, LDR3,LDR4, LDR5, LDR6,LDR7,LDR8}, segregated linear decision rules \cite{DRO-SLDR2, DRO-SLDR1}, piecewise linear decision rules (PLDR) \cite{PLDR1-NPhard, PLDR3-NPhard, PLDR2}, and polynomial decision rules \cite{PDR-2, PDR-QDR1}. Decision rules are attractive because they usually lead to tractable approximations for problems with fixed recourse. 
However, when the problem has random recourse, even the simplest linear decision rule is already intractable 
 \cite{Linear-decision-rule, Random-Recourse-NPhard}. 

We describe several improved decision rules relevant to the paper. Quadratic decision rules have been applied to obtain tighter approximations by solving conservative semidefinite programs \cite{PDR-QDR1,QDR-2} based on the approximate S-lemma \cite{QDR-Slemma}.  Quadratic decision rules are a class of polynomial decision rules with degree 2, which provide a decent trade-off between solution quality and computational cost. PLDR also provides a tighter approximation than basic LDR. The authors of \cite{PLDR-optimal} show that PLDR is optimal for two-stage adaptive robust optimization problems with fixed recourse when the uncertainty set is a polytope. Nevertheless, constructing the optimal piecewise affine policies is proven to be NP-hard \cite{PLDR4_NPhard}. The recent work \cite{two-stage-sampling-robust} proposes an approximation scheme based on overlapping PLDR for two-stage type-$\infty$ Wasserstein DRO, which can be solved as a linear program or a second-order conic program. The paper~\cite{DRO-solution-3} introduces the event-wise affine adaptations addressing solutions of two-stage DRO with Wasserstein and K-means (Voronoi regions) ambiguity sets. Although highly effective, these methods can only cope with two-stage DRO problems with fixed recourse. 

We now describe existing solution schemes for problems with random recourse. In the robust optimization setting, the paper \cite{MSRO-decision-rules} first derives copositive programming reformulations for the LDR approximations.  
Nevertheless, the resulting decision has the risk of being too conservative as it does not take into account distributional information obtained from  the historical data. At the other extreme, the  paper \cite{Random-Recourse-Stochastic} considers multi-stage stochastic optimization via LDR, where it requires complete knowledge of distributions of the random parameters.  The paper \cite{Wass-SDP} derives  copositive programming approximations for DRO problems with random recourse in the objective function. The reformulations, however, lead to huge semidefinite programming approximations that do not scale to large problem instances. Furthermore, the ambiguity set is based on the Wasserstein metric, whose finite-sample guarantees are known to suffer from the \emph{curse of dimensionality}~\cite{esfahani2018data}.\footnote{A recent result by Gao  \cite{gao2020finite} shows that under some Lipschitz continuity assumptions on the loss function, the curse of dimensionality can be eliminated. However, the assumptions are not satisfied in our two-stage setting because the loss function is not finite whenever the second-stage problem is infeasible. } 

To  the best of our knowledge, we are the first to propose an effective solution method to address two-stage data-driven DRO problems with random recourse. Specifically, we show that in view of the decision rules, the problems are amenable to exact copositive programming reformulations of polynomial size that admit tight, conservative approximations in semidefinite programming \cite{MSRO-decision-rules}. We further derive attractive out-of-sample performance guarantees on the solutions, which do not suffer from the curse of dimensionality. We summarize the paper’s main contributions, as follows:
\begin{enumerate}
    \item 	We apply the enhanced piecewise linear decision rules (PLDR) and piecewise quadratic decision rules (PQDR) to address the generic two-stage DRO problems using a partitioning scheme based on the Voronoi diagram, where the approximation quality is controlled by the number of partitions.  The traditional reformulation technique based on standard convex duality theory is not applicable because of the random recourse setting. We leverage the modern conic programming machinery to derive a concise copositive program (COP), which remains intractable, but admits excellent approximations in semidefinite programming.  
    \item As a byproduct of the proposed decision rule scheme, we construct a tailored ambiguity set with two layers of robustness: the conditional ambiguity sets for the partitions and the uncertainty set for  the marginal probabilities of falling inside different partitions. In view of the decision rules, the conditional ambiguity sets can aptly be defined through the conditional second moments of the random parameters. We then use the $\chi^2$-distance to construct the uncertainty set for the marginal probabilities. Combining the concentration inequalities for second moments and $\chi^2$ statistics, we derive the theoretical finite sample guarantee on the optimal solutions. 
    \item 
    We conduct extensive numerical experiments on network inventory allocation, newsvendor,  medical scheduling, and facility location problems.  We demonstrate that our copositive reformulations together with the inner semidefinite approximations achieve significantly better solution quality than the na\"ive sample-average approximation (SAA), especially with limited sample data. Even in the particular case of having random recourse present only in the objective, the experiment demonstrates that our method distinctly outperforms the state-of-the-art solution scheme proposed in \cite{Wass-SDP}. Additionally, as the number of samples and partitions increases, the performances of our solutions favorably get closer to the optimal ones.  
    \item 	We devise a decomposition algorithm that separates the complexity of optimizing for the first-stage decision, which constitutes a tractable second-order cone program, and the second-stage policy, which comprises several concise copositive programs. The second-stage subproblems are amenable to parallelized solutions that significantly reduce the runtime as the number of partitions increases. We prove that, under some regularity conditions, the algorithm converges in a finite number of iterations. 
    
\end{enumerate}
Two-stage DRO problems have been widely studied in the literature using different ambiguity sets. The paper \cite{DRO-NPhard} studies two-stage DRO problems with a non-data driven ambiguity set based on first- and second-order moments. In  \cite{DRO-solution-2}, the authors propose a modified SAA to approximate the two-stage DRO where the ambiguity set is a confidence region of a goodness-of-fit hypothesis test. The work in \cite{DRO-solution-1} derives an equivalent reformulation of two-stage optimization problems with an $L^{1}$-norm ambiguity set and uses SAA as an approximation. The paper \cite{bayraksan2015data} studies two-stage DRO problems using the $\phi$-divergence ambiguity sets, which lead to tractability but unfortunately unable to provide any finite-sample guarantees. 


 As a global optimization approach, the Benders decomposition algorithm is first proposed in \cite{benders-sp1, benders-sp2} as an efficient solution to the two-stage stochastic optimization problem with a discrete distribution. It utilizes the decomposition of the second-stage problem into separate subproblems for different scenarios, leading to a much faster computation of the recourse decision. The algorithm has been further developed to provide a globally optimal solution of the two-stage robust optimization problem, even with integer decision variables~\cite{benders-rp1, benders-rp2}. In a data-driven setting, the paper~\cite{benders-dro1} reformulates the two-stage DRO problem with Wasserstein metric to a two-stage robust optimization problem and applies the Benders algorithm as the solution scheme. The work \cite{benders-dro2} applies Benders-type  decomposition algorithms to the two-stage DRO problems with binary variables in both the first and second stages, as well as the two-stage linear DRO problems with continuous decision variables.

The rest of this article is organized as follows. In Section \ref{2-stage-DRO}, we derive the exact copositive programming reformulation after applying PLDR and PQDR to the two-stage DRO problems with random recourse. We derive the finite-sample performance guarantee for the solution 
in Section \ref{performance_guarantee}. Section \ref{decomposition_algorithm} develops a decomposition algorithm to solve the large-scale conic problem efficiently. 
Finally, Section \ref{experiments} reports numerical results and demonstrates the efficacy of the proposed method in terms of feasibility, performance, and computational effort. We discuss our conclusions in Section \ref{conclusion}. 

\subsection{Notation and terminology}
For any $\ell \in \mathbb{N}$, we define $[L]$ as the index set $\{1, \dots, L\}$. We denote $\mathbf{e}$ as the vector of all ones and  $\mathbf{e}_i$ as the $i$-th standard vector. The set of non-negative real numbers is denoted as $\RR_{+}$ and the set of vectors $\bm v \in \RR^{N}$ with non-negative elements is given by $\RR^{N}_{+}$. The indicator function of a subset $B$ of a set $A$ is denoted as $\mathbbm 1_{[B \in A]}$ and matrix $\bm 1$ refers to matrix of ones where every element equals to 1. We set the $p$-norm of a vector $\bm v$ as $\| \bm v \|_p$ and the $p$-norm of a matrix $\bm X$ as $\| \bm X \|_p$. We denote the Dirac delta measure as $\delta_{\bm\xi}$, which places a unit mass at $\bm\xi$. The space of symmetric matrices in $\RR^{N \times N}$ is defined as $\mathbb{S}^{N}$. The set of all probability measures supported on $\Xi$ is written as $\mathcal{P} (\Xi) \coloneqq \left\{ \mu \in\mathcal M_+: \int_{\Xi} \mu(d\xi) = 1 \right\}$, where $\mathcal M_+$ denotes the set of nonnegative Borel measures. The trace of a square matrix is described as $\tr(\bm X)$. We use $\bm X \geq \bm 0$ to indicate that $\bm X$ is a component-wise non-negative matrix. Notation $\textup{Rows}(\bm X) \in \mathcal{K}$ indicates that the vectors of rows of matrix $\bm X$ belong to cone $\mathcal{K}$.

The standard second-order cone is defined as $\mathcal{SOC} \in \RR^{S+1}$, i.e., $\bm v \in \mathcal{SOC} \Longleftrightarrow \|(v_1, \dots, v_S)^\top \| \leq~v_{S+1}$. We define the copositive cone as $\mathcal{C}(\RR^N_{+}) \coloneqq \left \{ \bm X \in \mathbb{S}^N: \bm v^\top \bm X \bm v \geq 0 \ \forall \bm v \in \RR^N_{+} \right \}$ and its dual cone the completely positive cone as $\mathcal{C}^*(\RR^N_{+}) \coloneqq \left \{ \bm X \in \mathbb{S}^N: \bm X = \sum_{i=1}^I \bm x^i(\bm x^i)^\top, \ \bm x^i \in \RR^N_{+} \right \}$ where $I$ is a positive integer with unspecified value. In this paper we study the generalized copositive cone $\mathcal{C(K)}$ and its dual cone the generalized completely positive cone $\mathcal{C^*(K)}$ for a closed and convex cone $\mathcal{K} \in \RR^N$. By specifying $\mathcal{K} = \RR_+$ we get $\mathcal{C}(\RR^N_{+})$ and $\mathcal{C}^*(\RR^N_{+})$ respectively. Here we refer to the linear optimization over $\mathcal{C(K)}$ and $\mathcal{C^*(K)}$ as a copositive program and a completely positive program, which is simplified notations of a generalized copositive program or a generalized completely positive program. For any symmetric matrix $\bm X \in \mathbb{S}^N$, $\bm X \succeq \bm 0$ denotes $\bm X$ is positive semidefinite. The notations $\bm X \succeq_{\mathcal{C(K)}} \bm 0$ and $\bm X \succeq_{\mathcal{C^*(K)}} \bm 0$ indicate that $\bm X$ is an element of $\mathcal{C(K)}$ and $\mathcal{C^*(K)}$ respectively, i.e., $\bm X \in \mathcal{C(K)}$ and $\bm X \in \mathcal{C^*(K)}$. In addition, the relation $\bm X \succ_{\mathcal{C(K)}}$ means $\bm X$ is  strictly copositive, i.e., $\bm v^\top \bm X \bm v > 0 \ \forall \bm v \in \mathcal{K}, \ \bm v \neq \bm 0$.

\section{Two-stage DRO with Decision Rules}
\label{2-stage-DRO}

We study the general risk-averse two-stage  linear distributionally robust  optimization problem (DRO) using the Conditional Value-at-Risk (CVaR) measure. In this adaptive optimization problem, a decision maker first selects a here-and-now decision $\bm x \in \X \in \RR^{N_1}$, which gives the immediate cost $\bm c^\top \bm x$. After the realization of the uncertain parameter vector $\bm \xi\in\Xi$, the wait-and-see decision $\bm y(\bm \xi) \in \RR^{N_2}$ that minimizes the second-stage cost $(\bm D \bm\xi)^\top \bm y(\bm \xi)$ is made. The decision maker seeks a safe decision $\bm x$ and policy $\bm y(\bm \xi)$ that perform the best in view of the worst-case CVaR  at level $\delta \in (0,1)$. The problem is formally written as
\begin{equation}
\label{eq:2s-dro1}
\begin{array}{ccl}
J^{\star} \coloneqq &\displaystyle \inf_{\bm x \in \mathcal{X}} & \bm c^\top \bm x + \displaystyle
\sup_{\PP\in\mathcal P} \PP\text{-CVaR}_\delta [Z(\bm x, \bm \xi)]. \\
\end{array}
\end{equation}
Here, $\mathcal P$ is an ambiguity set containing plausible distributions of the uncertain parameter $\bm \xi \in \RR^{S+1}$ with a known support $\Xi$. In this work, we will construct the ambiguity set in a data-driven manner using historical samples $\{ \hat{\bm\xi}_i\}_{i \in [N]}$ of the random vector $\bm\xi$. The min-max formulation \eqref{eq:2s-dro1} implies that the first-stage decision~$\bm x$ is chosen considering the most adverse distribution $\PP\in\mathcal P$ that maximizes the CVaR of the second-stage recourse function $Z(\bm x, \bm \xi)$. Note that when $\delta = 1$, the worst-case CVaR reduces to the worst-case expectation. 

The second-stage recourse function $Z(\bm x, \bm \xi)$ in $\eqref{eq:2s-dro1}$ is defined as the optimal value of the linear program
\begin{equation}
\label{eq:second_stage_value}
\begin{array}{ccl}
Z(\bm x, \bm \xi) \coloneqq & \displaystyle \inf & (\bm D \bm\xi)^\top \bm y \\
& \st & \bm y \in \RR^{N_2}\\
& & \bm T_\ell(\bm x)^\top\bm \xi  \leq \bm (\bm W_\ell\bm\xi)^\top\bm y \quad \forall \ell \in [L],
\end{array}
\end{equation}
where $\bm D \in \RR^{N_2 \times (S+1)}, \ \bm W_\ell \in \RR^{N_2 \times (S+1)}, \ \ell \in [L]$. The matrices $\bm T_\ell(\bm x) \in \RR^{S+1}, \ \ell \in [L]$, are assumed to be affine functions in $\bm x$. Two-stage linear DRO problems  
are NP-hard even with fixed recourse~\cite{DRO-NPhard}. 
Our general structure above involves random recourse in both the objective $(\bm D \bm\xi)^\top \bm y(\bm \xi)$ and the constraints $(\bm W_\ell\bm\xi)^\top\bm y(\bm \xi)$, where the adaptive decision $\bm y(\bm \xi)$ is multiplied  with the uncertain parameter $\bm \xi$. The random recourse introduces significant challenges in addressing the problem. In this paper, we employ the decision rule approach that enables an exact reformulation with tractable approximations. To this end, we first reformulate problem~$\eqref{eq:2s-dro1}$ as a distributionally robust semi-infinite linear program involving a worst-case expectation in the objective function.

\begin{proposition}
The risk-averse two-stage DRO problem \eqref{eq:2s-dro1} can be reformulated as
\begin{equation}
\label{eq:2s-dro2}
\begin{array}{clll}
\displaystyle \inf & \displaystyle \bm c^\top \bm x + \theta + \frac{1}{\delta} \sup_{\PP\in\mathcal P} \EE_\PP [\tau(\bm \xi)] \\
\st & \bm x \in \X, \ \theta \in \RR, \ \bm y: \RR^{S+1} \rightarrow \RR^{N_2}, \ \tau: \RR^{S+1} \rightarrow \RR\\
& \left. 
\begin{aligned}
& \tau(\bm \xi) \geq 0 \\
& \tau(\bm \xi) \geq (\bm D \bm \xi)^\top \bm y (\bm \xi) - \theta \\
& \bm T_\ell(\bm x)^\top\bm \xi  \leq \bm (\bm W_\ell \bm\xi)^\top\bm y(\bm \xi) \quad \forall \ell \in [L] \\
\end{aligned} \right \} \forall \bm\xi \in \Xi,
\end{array}
\end{equation}
where the second-stage decision variables $\bm y$ and $\tau$ are measurable mappings from $\RR^{S+1}$ to $\RR^{N_2}$ and  from $\RR^{S+1}$ to $\RR$, respectively.
\end{proposition}

\begin{proof}
See the Appendix A.1.
\end{proof}
\noindent The reformulated problem \eqref{eq:2s-dro2} contains random recourse terms in the constraints even if the original two-stage DRO problem \eqref{eq:2s-dro1} only has a random recourse term in the objective. In the following, we focus on addressing the problem \eqref{eq:2s-dro2}. 

Before we describe the proposed decision rule approach, we briefly explain the support set below. The uncertain parameter vector $\bm \xi$ belongs to a support set $\Xi$ defined as a slice of a convex cone $\mathcal{K} \in \RR^{S}\times \RR_{+} $, given by

\begin{equation}
\label{eq:uncertainty-set}
    \Xi \coloneqq \left\{\bm \xi \coloneqq \begin{bmatrix}\bm \zeta\\ \nu \end{bmatrix} \in \mathcal{K}:
    \nu = 1
     \right\}.
\end{equation}

\noindent The restriction of $\bm \xi$ as $[\bm \zeta, 1]^\top$ enables us to simplify any affine function of the primitive parameter vector~$\bm \zeta$ to a linear function of $\bm \xi$. Similarly, it can also represent any quadratic function in a homogenized form. 
 The cone $\mathcal{K}$ is assumed to satisfy the following mild condition.
\begin{assume}\label{assum1}
The cone $\mathcal{K}$ is nonempty, compact, convex and full-dimensional.
\end{assume}
\noindent The support set $\Xi$ in \eqref{eq:uncertainty-set} can model widely used support sets. For instance, we can define a polytope support set by setting 

$$\mathcal{K} \coloneqq \left \{\bm \xi \coloneqq \begin{bmatrix}\bm \zeta\\ \nu \end{bmatrix} \in \RR^{S} \times \RR_+: \bm P \bm \zeta \geq \bm t \nu \right \},$$ with $\bm P \in \RR^{S_p \times S}, \ \bm t \in \RR^{S_p}$. In addition, we can model an ellipsoid or 2-norm ball support set by setting 
$$ \mathcal{K} \coloneqq \left \{\bm \xi \coloneqq \begin{bmatrix} \bm \zeta\\ \nu \end{bmatrix} \in \RR^{S} \times \RR_+: \| \mathbf{R} \bm \zeta \|_2 \leq q \nu \right \},$$
with $\bm R \in \RR^{S_r \times S},\ q \in \RR$.

\subsection{The decision rule framework}
\label{sec:PLDR}
In view of the random recourse setting, we adopt the combination of  piecewise linear decision rules (PLDR) and piecewise quadratic decision rules (PQDR) to conservatively approximate the two-stage problem \eqref{eq:2s-dro1} by restricting the adaptive decisions  $\bm y (\bm \xi)$ and $\tau(\bm \xi)$ to piecewise affine functions and piecewise quadratic functions in the uncertain parameters, respectively. 
To this end, we partition the support set $\Xi$ into $K$ subsets $\Xi_1$,\ldots, $\Xi_K$, and we optimize basic linear or quadratic decision rules in each partition separately. Specifically, the PLDR scheme for the decision variable $\bm y (\bm \xi)$ is given by
\begin{equation*}
\bm y(\bm \xi) = \bm Y_k\bm \xi \quad \forall \bm \xi \in \Xi_k \ \forall k \in [K],
\end{equation*}
 where $\bm Y_k \in \mathbb{R}^{N_2 \times (S+1)}$ is the coefficient of the linear decision rules for partition $k$. 
 
 Next,  we observe in problem \eqref{eq:2s-dro2} that at optimality the second-stage epigraphical variable $\tau(\bm \xi)$ coincides with $\max \left \{ (\bm D \bm \xi)^\top \bm y (\bm \xi) - \theta, 0 \right \}$. The random recourse term $(\bm D \bm \xi)^\top \bm y (\bm \xi)$ constitutes a piecewise quadratic function in $\bm \xi$ after applying PLDR to~$\bm y(\bm \xi)$, and the semi-infinite inequality constraints in~\eqref{eq:2s-dro2} are satisfied if $\tau(\bm \xi)$ exhibits a piecewise quadratic form in $\bm \xi$. Thus, we apply PQDR to the decision variable~$\tau(\bm \xi)$ as 
 \begin{equation*}
 \tau(\bm \xi) = \bm \xi^\top \bm Q_k \bm\xi \quad \forall \bm \xi \in \Xi_k \ \forall k \in [K],
\end{equation*}
where $ \bm Q_k \in \mathbb{R}^{(S+1) \times (S+1)}$ is the coefficient of the quadratic decision rules for partition $k$.

Employing the combination of PLDR and PQDR and applying the law of total expectation yield the conservative piecewise decision rule (PDR) problem
\begin{equation}
\label{eq:pdr}
\begin{array}{ccl}
J^{\textup{PDR}} \coloneqq & \displaystyle \inf & \displaystyle \bm c^\top \bm x + \theta + \frac{1}{\delta} \sup_{\PP\in\mathcal P} \sum_{k \in [K]}  \PP(\bm \xi \in \Xi_k)\EE_\PP[\bm \xi^\top \bm Q_k \bm\xi | \bm\xi\in\Xi_k]  \\
&\st & \bm x \in \X, \ \theta \in \RR, \ \bm Y_k \in \RR^{N_2 \times (S+1)},\ \bm Q_k \in \RR^{(S+1) \times (S+1)}\\
& & \bm{\mathcal{T}}_\ell (\bm x)^\top\bm \xi  \leq (\bm{\mathcal{W}}_\ell \bm\xi)^\top\bm Y_k\bm\xi + \lambda_\ell \bm \xi^\top \bm Q_k \bm \xi + \kappa_\ell \theta \quad \forall \bm\xi \in \Xi_k \ \forall k \in [K] \ \forall \ell \in [L+2],
\end{array}
\end{equation}
where 
\begin{equation*}
\left.\begin{matrix}
\bm{\mathcal{T}}_\ell (\bm x) = \bm T_\ell(\bm x), & \bm{\mathcal{W}}_\ell = \bm W_\ell, & \lambda_\ell = 0, & \kappa_\ell = 0 \quad & \forall \ell \in [L], \\ 
\bm{\mathcal{T}}_\ell (\bm x) = [\bm 0], & \bm{\mathcal{W}}_\ell = [\bm 0], & \lambda_\ell = 1, & \kappa_\ell = 0 \quad & \ell = L+1, \\ 
\bm{\mathcal{T}}_\ell (\bm x) = [\bm 0], & \bm{\mathcal{W}}_\ell = -\bm D, & \lambda_\ell = 1, & \kappa_\ell = 1 \quad & \ell \in L+2. \\ 
\end{matrix}\right.
\end{equation*}

In this work, we construct the partitions of the support set $\Xi$ according to Voronoi regions. Starting with a set of constructor points $\{\bm \xi_k'\}_{k \in [K]}$ obtained independently of the samples $\{ \hat{\bm\xi}_i\}_{i \in [N]}$, we define the partition $\Xi_k$ as the set of all points in $\Xi$ whose Euclidean distance is closer to $\bm \xi_k'$ than any other constructor points. That is to say, for the $k^{\textup{th}}$ partition we have
\begin{equation*}
\begin{array}{rl}
\Xi_k =& \left \{ \bm \xi \in \Xi:\left \| \bm \xi - \bm \xi_k' \right \|_2 \leq \left \| \bm \xi - \bm \xi_i' \right \|_2 \ \forall i \in [K]:i \neq  k \right \} \\
=& \left \{ \bm \xi \in \Xi: 2(\bm \xi_i' - \bm \xi_k')^\top \bm  \xi \leq \bm  \xi_i'^\top\bm \xi_i'  - \bm \xi_k'^\top \bm  \xi_k' \ \forall i \in [K]:i \neq  k \right \} \\
=& \left \{ \bm \xi \in \mathcal{K}_k:
    \mathbf{e}_{S+1}^\top \bm \xi = 1
    \right \},
\end{array}
\end{equation*}
where 
\begin{equation}
\label{eq:Voronoi_partitions_cones}
 \mathcal{K}_k = \left \{ \bm \xi \in \mathcal{K} : 2(\bm \xi_i' - \bm \xi_k')^\top \bm  \xi \leq \bm  \xi_i'^\top\bm \xi_i' - \bm \xi_k'^\top \bm  \xi_k', \ \forall i \in [K]:i \neq  k \right \}
\end{equation}
is a convex cone describing the partition.

\subsection{Distributionally robust model}

The objective function of problem \eqref{eq:pdr} contains two sets of terms: the partition probabilities $\PP(\bm \xi \in \Xi_k)$, $k \in [K]$, and the conditional expectations $\EE_\PP[\bm \xi^\top \bm Q_k \bm\xi | \bm \xi \in \Xi_k]$, $k \in [K]$, whose values depend on the distribution of the uncertain parameter $\bm \xi$. These terms can principally be 
estimated using the empirical distribution $\hat{\PP} = \frac{1}{N} \sum_{i \in [N]} \delta_{\hat{\bm\xi}_i}$ assigning equal mass $1/N$ to all the historical samples. However, this leads to estimation errors and may yield inferior decisions, which perform poorly in practice. To mitigate the overfitting effects,  we design a data-driven ambiguity set with two layers of robustness on the partition probabilities and the conditional expectations. 

\begin{defi} 
\textup{(Ambiguity set).} 
We define the ambiguity set $\mathcal P$ as
\begin{equation}
\label{eq:ambiguity_set}
\mathcal P \coloneqq \left\{\PP \in \mathcal{P} (\Xi):\PP = \sum_{k \in [K]} p_k \PP_k \ \textup{  with  } \  \bm p \in \Delta \textup{ and } \ \PP_k \in \mathcal{P}_k \quad \forall k \in [K]
\right\}.
\end{equation}
The uncertainty set $\Delta$ for the partition probability vector $\bm p$ is defined in terms of $\chi^2$-distance as
\begin{equation}
\label{eq:ambiguity_set-prob}
\Delta \coloneqq \left\{\bm p \in \RR^K_{+} : \mathbf{e}^\top \bm p = 1, \ \sum_{k = 1}^{K} (p_k - \hat{p}_k)^2/p_k \leq \gamma
\right\},
\end{equation}
where $\hat{p}_k = \tfrac{1}{N}\sum_{i \in [N]}\mathbbm{1}(\hat{\bm \xi}_i \in \Xi_k)$, $k \in [K]$. And the ambiguity sets $\mathcal P_k, \ k \in [K]$, for the conditional distributions $\PP_1, \dots, \PP_K$ are given by
\begin{equation}
\label{eq:ambiguity_set-k}
\mathcal P_k \coloneqq \left\{\PP_k \in \mathcal{P} (\Xi_k): \left\| \EE_{\PP_k} [\bm\xi\bm\xi^\top] - \EE_{{\hat{\PP}}_k} [\bm\xi\bm\xi^\top] \right\|_F \leq \epsilon_k
\right\},
\end{equation}
where $\mathcal{I}_k = \left\{ i \in [N]: \hat{\bm \xi}_i \in \Xi_k\right \}$ and $\hat{\PP}_k = \frac{1}{|\mathcal{I}_k|} \sum_{i \in \mathcal{I}_k} \delta_{\hat{\bm \xi}_i}$, $k\in[K]$.
\end{defi}

Any distribution $\PP$ in the ambiguity set $\mathcal P$ is defined as a mixture of the conditional distributions $\PP_1,\ldots,\PP_K$ with weights $p_1,\ldots,p_K$, respectively. The weights are assumed to belong to the uncertainty set $\Delta$. The uncertainty set contains all probability vectors $\bm p$ sufficiently close to the empirical marginal probability vector $\hat{\bm p}$ in terms of the $\chi^2$-distance, which belongs to the class of $\phi$-divergences \cite{phi-divergence}. The set is described by a single parameter $\gamma$ whose value can  theoretically be determined so as to attain a certain out-of-sample performance guarantee. 
Each conditional distribution $\PP_k$ belongs to the ambiguity set $\mathcal P_k$, which is defined as the set of all probability distributions whose second-moment matrix $\EE_{\PP_k} [\bm\xi\bm\xi^\top]$ is within a distance $\epsilon_k$  from the empirical second-moment matrix $\EE_{{\hat{\PP}}_k}[\bm\xi\bm\xi^\top]$ with respect to the Frobenius norm. 
Similarly to $\gamma$, the values of the parameters $\epsilon_1,\ldots,\epsilon_K$ can be theoretically determined to achieve the desired performance guarantees. 

With the ambiguity set $\mathcal{P}$ defined in $\eqref{eq:ambiguity_set}$, the objective function of problem \eqref{eq:pdr} can be rewritten as 
\begin{align}
& \bm c^\top \bm x + \theta + \frac{1}{\delta} \displaystyle \sup_{\PP\in\mathcal P} \sum_{k \in [K]} \PP(\bm \xi \in \Xi_k) \EE_\PP [\bm \xi^\top \bm Q_k \bm\xi | \bm\xi \in \Xi_k] \nonumber  \\
= \ & \bm c^\top \bm x + \theta + \frac{1}{\delta} \displaystyle \sup_{\bm p \in \Delta}  \sup_{\PP_k \in \mathcal{P}_k \forall k} \sum_{k \in [K]} p_k \EE_{\PP_k} [\bm \xi^\top \bm Q_k \bm\xi] \nonumber \\
= \ &  \bm c^\top \bm x + \theta + \frac{1}{\delta} \displaystyle \sup_{\bm p \in \Delta} \sum_{k \in [K]}  p_k \sup_{\PP_k \in \mathcal{P}_k}  \EE_{\PP_k} [\bm \xi^\top \bm Q_k \bm\xi].\label{eq:pdr-obj}
\end{align}
The second equality holds because each innermost maximization optimizes the conditional distribution $\PP_k$ separately. 
In the following, we will first derive the reformulation of the $k^\textup{th}$ inner worst-case expectation over the conditional ambiguity set $\mathcal P_k$ defined in $\eqref{eq:ambiguity_set-k}$. Then, we derive the reformulation of the outer worst-case expectation problem over the uncertainty set $\Delta$.
\begin{proposition}\label{prop2}
With the conditional ambiguity set $\mathcal P_k$ defined in $\eqref{eq:ambiguity_set-k}$, the inner worst-case expectation can be reformulated as
\begin{equation}
\label{eq:inner_ambiguity_set_ref}
\begin{array}{ccll}
\displaystyle \sup_{\PP_k \in \mathcal{P}_k}  \EE_{\PP_k} [\bm \xi^\top \bm Q_k \bm\xi] = & \displaystyle \inf & \alpha_k + \tr(\bm Q_k\hat{\bm\Omega}_k) + \tr(\bm B_k\hat{\bm\Omega}_k) + \epsilon_k\|\bm Q_k^\top + \bm B_k\|_F  \\
& \st & \alpha_k \in \RR, \ \bm B_k \in \mathbb{S}^{S+1}\\
& & \alpha_k + \bm \xi^\top \bm B_k \bm \xi \geq  0  \quad \forall \bm\xi \in \Xi_k,
\end{array}
\end{equation}
where $\hat{\bm \Omega}_k = \frac{1}{|\mathcal{I}_k|} \sum_{i \in \mathcal{I}_k} \hat{\bm \xi}_i \hat{\bm \xi}_i^\top$. 
\end{proposition}

\begin{proof}
By the cyclic property of the trace operation, we have $\bm\xi^\top \bm Q_k\bm\xi=\tr(\bm\xi^\top \bm Q_k\bm\xi)=\tr(\bm Q_k\bm\xi\bm\xi^\top )$. Then, using the linearity of expectation, we can reexpress the expectation $\EE_{\PP_k} [\tr(\bm Q_k\bm\xi\bm\xi^\top)]$ as $\tr\left( \bm Q_k \EE_{\PP_k} [\bm\xi \bm \xi^\top]\right)$. 
Substituting the definition of the conditional ambiguity set $\mathcal P_k$, we can rewrite the resulting worst-case expectation as the moment problem
\begin{equation}
\label{eq:inner_ambiguity_set_ref0}
\begin{array}{ccccl}
 \displaystyle\sup_{\PP_k \in \mathcal{P}_k} \tr\left( \bm Q_k \EE_{\PP_k} [\bm\xi \bm \xi^\top]\right) &= &  & \sup & \tr\left( \bm Q_k \bm \Omega_k \right)\\
& & & \st & \mu \in \mathcal{M}_+, \ \bm \Omega_k \in \mathbb{S}^{S+1}\\
& & & & \int_{\Xi_k} \mu (\mathrm{d} \bm \xi) = 1 \\
& & & & \bm \Omega_k = \int_{\Xi_k}\bm \xi \bm \xi^\top \mu(\mathrm{d} \bm \xi)\\
& & & & \| \bm \Omega_k - \hat{\bm \Omega}_k \|_F \leq \epsilon_k.
\end{array}
\end{equation}
The dual of problem $\eqref{eq:inner_ambiguity_set_ref0}$ is given by problem $\eqref{eq:inner_ambiguity_set_ref}$. Since there always exists an interior  point (e.g., $\alpha_k > 0$, $\bm B_k = \mathbf{0}$) in problem $\eqref{eq:inner_ambiguity_set_ref}$, strong duality holds, which implies that the optimal objective values of problem~$\eqref{eq:inner_ambiguity_set_ref0}$ and $\eqref{eq:inner_ambiguity_set_ref}$ coincide. This completes the proof.  
\end{proof}

\begin{proposition} \textup{(\cite[Theorem 4.1]{delta-reformulation}).} \label{prop3}
With the uncertainty set $\Delta$ described in $\eqref{eq:ambiguity_set-prob}$, for any $\bm \varphi \in \RR^K$, the worst-case expectation problem $\max_{\bm p \in \Delta} \bm \varphi^\top \bm p$ is equivalent to the second-order cone program (SOCP)
\begin{equation}
\label{eq:outer_ambiiguityp_set_ref}
\begin{array}{ccll}
& \displaystyle \max &  \bm \varphi^\top \bm p \\
& \st & \bm p, \bm q \in \RR^{K}_{+}, \ \mathbf{e}^\top \bm p  = 1, \ \mathbf{e}^\top \bm q \leq \gamma \\
& & \sqrt{(p_k -  \hat{p}_k)^2 + \frac{1}{4}p_k^2 + q_k^2} \leq \frac{1}{2}p_k + q_k \quad \forall  k \in [K],
\end{array}
\end{equation}
i.e., the optimal value and optimal solution can be computed by solving problem $\eqref{eq:outer_ambiiguityp_set_ref}$. The optimal value can also be obtained by solving the following SOCP dual to $\eqref{eq:outer_ambiiguityp_set_ref}$:
\begin{equation*}
\begin{array}{ccll}
& \displaystyle \min &  \gamma \omega - \eta - 2\hat{\bm p}^\top \bm r + 2\omega\hat{\bm p}^\top \mathbf{e} \\
& \st & \omega \in \RR_{+}, \ \eta \in \RR, \ \bm r, \bm s \in \RR^{K}\\
& & \varphi_k \leq s_k, \ s_k + \eta \leq \omega, 
\ \sqrt{4r_k^2 + (s_k +  \eta)^2} \leq 2\omega - s_k -\eta \quad \forall  k \in [K].
\end{array}
\end{equation*}
\end{proposition}

Combining Propositions \ref{prop2} and  \ref{prop3}, and identifying $\varphi_k$ with the optimal value of the $k^\textup{th}$ inner worst-case expectation $\sup_{\PP_k \in \mathcal{P}_k}  \EE_{\PP_k} [\bm \xi^\top \bm Q_k \bm\xi]$, we obtain the reformulation of the objective function $\eqref{eq:pdr-obj}$. The PDR problem~$\eqref{eq:pdr}$ is thus equivalent to the following semi-infinite program:
\begin{equation}
\label{eq:pdr-ref-1}
\begin{array}{ccll}
J^{\textup{PDR}} = &  \displaystyle \inf & \bm c^\top \bm x + \theta + \frac{1}{\delta} ( \gamma \omega - \eta - 2\hat{\bm p}^\top \bm r + 2\omega\hat{\bm p}^\top \mathbf{e}) \\
& \st & \bm x \in \X, \  \theta, \eta \in \RR,\ \omega \in \RR_{+},\ \bm r, \bm s, \bm \alpha \in \RR^{K}, \\ 
& & \left.
\begin{aligned}
& \bm Y_k \in \RR^{N_2 \times (S+1)},\ \bm Q_k \in \RR^{(S+1) \times (S+1)}, \ \bm B_k \in \mathbb{S}^{S+1} \\
& s_k + \eta \leq \omega, 
\sqrt{4r_k^2 + (s_k +  \eta)^2} \leq 2\omega - s_k -\eta\\
& \alpha_k + \tr(\bm Q_k\hat{\bm\Omega}_k) + \tr(\bm B_k\hat{\bm\Omega}_k) + \epsilon_k\|\bm Q_k^\top +  \bm B_k\|_F \leq s_k  \\
& \bm{\mathcal{T}}_\ell (\bm x)^\top\bm \xi  \leq (\bm{\mathcal{W}}_\ell \bm\xi)^\top\bm Y_k\bm\xi + \lambda_\ell \bm \xi^\top \bm Q_k \bm \xi + \kappa_\ell \theta  \quad \forall \ell \in [L+2] \ \forall \bm\xi \in \Xi_k \\
& \alpha_k + \bm \xi^\top \bm B_k \bm \xi \geq  0  \quad \forall \bm\xi \in \Xi_k  \\
\end{aligned} \right \} \forall k \in [K].
\end{array}
\end{equation}


\subsection{Copositive reformulations}
\label{copositive_reformulations}


Problem \eqref{eq:pdr-ref-1} has infinitely many constraints parametrized by the realizations of $\bm \xi$ in the partitions $\Xi_k, \ k \in [K]$. For any fixed decision $\left( \bm x, \theta, \bm \alpha, \bm Y_k, \bm Q_k , \bm B_k\right )$, the constraints are equivalent to the following constraints involving nonconvex quadratic minimization problems: 
\begin{subequations}
\label{eq:pdr_cons}
\begin{align}
    \displaystyle 0\leq\inf_{\bm \xi \in  \Xi_{k}} & \; \bm \xi^\top \bm{\mathcal{W}}^\top_\ell \bm Y_k \bm\xi + \lambda_\ell \bm \xi^\top \bm Q_k \bm \xi - \bm{\mathcal{T}}_\ell (\bm x)^\top \bm\xi + \kappa_\ell \theta \quad \forall \ell \in [L+2] \ \forall k \in [K]  \label{eq:pdr_cons1} \\
    \displaystyle 0 \leq \inf_{\bm \xi \in  \Xi_{k}} &\; \bm \xi^\top \bm B_k \bm \xi + \alpha_k \quad \forall k \in [K]. \label{eq:pdr_cons2}
\end{align}
\end{subequations}
To convexify the problems, we utilize the copositive programming scheme to derive equivalent reformulations of \eqref{eq:pdr_cons}. We first introduce the following lemmas to establish the equivalence between a nonconvex quadratic optimization problem and its corresponding copositive formulation.

\begin{lemma} \textup{(\cite[Lemma 4]{minimum-ellipsoids}).}
\label{lem:interior-point}
With the $k^\textup{th}$ partition 
defined as 
\begin{equation*}
    \Xi_k \coloneqq \left\{\bm \xi \coloneqq \begin{bmatrix}\bm \zeta\\ \nu \end{bmatrix} \in \mathcal{K}_k:
    \nu = 1
     \right\},
\end{equation*}
there exist $\bm \Lambda \in \mathbb{S}^S$ and $\bm \zeta \in \mathbb{R}^S$ such that 
\begin{equation*}
    \begin{bmatrix}
\bm \Lambda & \bm \zeta\\ 
 \bm \zeta^\top & 1 
\end{bmatrix} \succ_{\mathcal{C^*}(\mathcal{K}_k)} \bm 0.
\end{equation*}
\end{lemma}

\begin{lemma} \textup{(\cite[Corollary 8.3]{copositive-program}).} \label{lem2}
The nonconvex quadratic program given by
\begin{equation}
\label{eq:qua-progran}
\begin{array}{clll}
 & \displaystyle \sup_{\bm \xi \in \RR^{S+1}} & \bm \xi^\top \hat{\bm C}_0 \bm\xi \\
 &\st & \mathbf{e}_{S+1}^\top \bm \xi = 1\\
 & & \bm\xi \in  \mathcal K_k, \\
 \end{array}
\end{equation}
and the completely positive program (CPP)
\begin{equation}
\label{eq:cpp}
\begin{array}{clll}
 & \displaystyle \sup & \left \langle \hat{\bm C}_0, 
 \begin{bmatrix}
\bm \Lambda & \bm \zeta\\ 
 \bm \zeta^\top & 1 
\end{bmatrix} \right \rangle \\
 &\st & 
\begin{bmatrix}
\bm \Lambda & \bm \zeta\\ 
 \bm \zeta^\top & 1 
\end{bmatrix} \in \mathbb{S}^{S+1}, 
\begin{bmatrix}
\bm \Lambda & \bm \zeta\\ 
 \bm \zeta^\top & 1 
\end{bmatrix} \in \mathcal{C^*}(\mathcal{K}_k)\\
 \end{array}
\end{equation}
are equivalent, i.e., (i) problem \eqref{eq:qua-progran} and problem \eqref{eq:cpp} share the same optimal value; (ii) if 
$$\begin{bmatrix}
\bm \Lambda^{\star} & \bm \zeta^{\star}\\ 
(\bm \zeta^{\star})^\top & 1 
\end{bmatrix}$$
is an  optimal solution for \eqref{eq:cpp}, then $[\bm \zeta^{\star}, 1]^\top$ is in the convex hull of optimal solutions for \eqref{eq:qua-progran}.
\end{lemma}

Based on the duality of conic programming, the dual of the CPP \eqref{eq:cpp} is given by the following COP with respect to $\mathcal{K}_k$:
\begin{equation}
\label{eq:cop}
\begin{array}{cll}
 & \displaystyle \inf_{u \in \RR} & u\\
 &\st & u \mathbf{e}_{S+1} \mathbf{e}_{S+1}^\top 
 - \bm{\hat{C}}_0 \in \mathcal{C} (\mathcal K_k). \\
\end{array}
\end{equation}

\begin{lemma} \label{lem3}
Strong duality holds between problems \eqref{eq:cpp} and \eqref{eq:cop} under Assumption 1.
\end{lemma}
\begin{proof}
The proof follows from Lemma \ref{lem:interior-point}, which establishes a Slater point for the primal CPP \eqref{eq:cpp}. 
\end{proof}
Using the above lemmas, we  now derive the copositive reformulations of the constraints in \eqref{eq:pdr_cons}. In the following, for notational convenience, we define the matrices $\bm \Delta^k_\ell(\bm x,\bm Y_k,\bm Q_k)$, $\ell \in [L+2]$, $k\in[K]$, that are affine in their arguments, as follows:
\begin{equation}
\label{eq:pldr_omega}
\begin{array}{clll}
\bm \Delta^k_\ell(\bm x,\bm Y_k,\bm Q_k) \coloneqq \frac{1}{2}\left ( \bm{\mathcal{W}}^\top_\ell \bm Y_k +\bm Y_k^\top \bm{\mathcal{W}}_\ell + \lambda_\ell \bm Q_k + \lambda_\ell \bm Q_k^\top - \bm{\mathcal{T}}_\ell(\bm x)\mathbf{e}_{S+1}^\top - \mathbf{e}_{S+1} \bm{\mathcal{T}}_\ell^\top(\bm x) \right ).
\end{array}
\end{equation}

\begin{proposition}
\label{prop:constraint_reformulation}
The constraints in \eqref{eq:pdr_cons1} are satisfied if and only if there exist $\pi^k_\ell \in \RR$, $\ell \in [L+2]$, $\forall k \in [K]$, such that
\begin{equation}
\label{eq:pdr_cons1_1}
\begin{array}{clll}
\bm \Delta^k_\ell(\bm x,\bm Y_k,\bm Q_k) - \pi^k_\ell \mathbf{e}_{S+1} \mathbf{e}_{S+1}^\top 
\in \mathcal{C} (\mathcal K_k), \; \pi^k_\ell + \kappa_\ell \theta \geq 0  \quad \forall \ell \in [L+2] \ \forall k \in [K].
\end{array}
\end{equation}
Similarly, the constraints in $\eqref{eq:pdr_cons2}$ are satisfied if and only if there exists $\bm \rho \in \RR^{K}$ such that 
\begin{equation}
\label{eq:pdr_cons2_1}
    \bm B_k - \rho_k \mathbf{e}_{S+1} \mathbf{e}_{S+1}^\top \in \mathcal{C} (\mathcal K_k), \ \rho_k + \alpha_k \geq 0 \quad \forall k \in [K].
\end{equation}
\end{proposition}

\begin{proof}
The constraint corresponding to indices $\ell$ and $k$ in \eqref{eq:pdr_cons1} can be rewritten as 
\begin{equation}
\label{eq:pdr_cons1_2}
\begin{array}{ccl}
 - \kappa_\ell \theta \leq & \displaystyle \inf_{\bm \xi \in \RR^{S+1}} & \bm \xi^\top \bm{\mathcal{W}}^\top_\ell \bm Y_k \bm\xi + \lambda_\ell \bm \xi^\top \bm Q_k \bm \xi - \bm{\mathcal{T}}_\ell (\bm x)^\top \bm\xi \\
 &\st & \mathbf{e}_{S+1}^\top \bm \xi = 1\\
 & & \bm\xi \in  \mathcal K_k. \\
\end{array} 
\end{equation}
According to Lemma \ref{lem2}, the quadratic minimization problem on the right-hand side of problem $\eqref{eq:pdr_cons1_2}$ is equivalent to the following linear program over the cone of completely positive matrices
\begin{equation}
\label{eq:pdr_cons1_3}
\begin{array}{clll}
 & \displaystyle \inf & \left \langle \bm \Delta^k_\ell(\bm x,\bm Y_k,\bm Q_k), \begin{bmatrix}
\bm \Lambda^k_\ell & \bm \zeta^k_\ell\\ 
 {\bm \zeta^k_\ell}^\top & 1 
\end{bmatrix} \right \rangle \\
 &\st & 
\begin{bmatrix}
\bm \Lambda^k_\ell & \bm \zeta^k_\ell\\ 
 {\bm \zeta^k_\ell}^\top & 1 
\end{bmatrix} \in \mathbb{S}^{S+1}, \; \begin{bmatrix}
\bm \Lambda^k_\ell & \bm \zeta^k_\ell\\ 
 {\bm \zeta^k_\ell}^\top & 1 
\end{bmatrix} \in \mathcal{C^*} (\mathcal K_k).
\end{array}
\end{equation}
The dual  of the above CPP \eqref{eq:pdr_cons1_3} is given by
\begin{equation}
\label{eq:pdr_cons1_4}
\begin{array}{clll}
 & \displaystyle \sup_{\pi^k_\ell  \in \RR} & \pi^k_\ell\\
 &\st &  \bm \Delta^k_\ell(\bm x,\bm Y_k,\bm Q_k) - \pi^k_\ell\bm e_{S+1} \bm e_{S+1}^\top 
 \in \mathcal{C} (\mathcal K_k). \\
\end{array}
\end{equation}
In view of the strong duality result between problems \eqref{eq:pdr_cons1_3} and \eqref{eq:pdr_cons1_4} in Lemma \ref{lem3}, the optimal value of the above problem coincides with the optimal value of the right-hand-side problem in \eqref{eq:pdr_cons1_2}. Under Assumption~1, the constraint in  \eqref{eq:pdr_cons1_2} is satisfied if and only if there exists a scalar $\pi_\ell^k$ such that $\pi^k_\ell + \kappa_\ell \theta \geq 0$ and
\begin{equation*}
\begin{array}{clll}
\bm \Delta^k_\ell(\bm x,\bm Y_k,\bm Q_k) - \pi^k_\ell \mathbf{e}_{S+1} \mathbf{e}_{S+1}^\top 
\in \mathcal{C} (\mathcal K_k).
\end{array}
\end{equation*}
Performing the same reformulations for all $(L+2)\times K$ constraints in \eqref{eq:pdr_cons1} leads to the constraint system in~\eqref{eq:pdr_cons1_1}. The reformulations for \eqref{eq:pdr_cons2} proceed in a similar way. Thus, the claim follows. 
\end{proof}
Proposition \ref{prop:constraint_reformulation} enables us to reformulate the semi-infinite program $\eqref{eq:pdr-ref-1}$ as a COP as follows.

\begin{theorem}
The PDR problem \eqref{eq:pdr} is equivalent to the following polynomial-size copositive program 
\begin{equation}
\label{eq:pdr_ref}
\begin{array}{ccll}
J^{\textup{PDR}} = &  \displaystyle \inf & \bm c^\top \bm x + \theta + \frac{1}{\delta} ( \gamma \omega - \eta - 2\hat{\bm p}^\top \bm r + 2\omega\hat{\bm p}^\top \mathbf{e}) \\
& \st & \bm x \in \X, \  \theta, \eta \in \RR,\ \omega \in \RR_{+},\ \bm r, \bm s, \bm \alpha \in \RR^{K} \\
& & \bm \pi_k \in \RR^L, \ \bm Y_k \in \RR^{N_2 \times (S+1)},\ \bm Q_k \in \RR^{(S+1) \times (S+1)}, \ \bm B_k \in \mathbb{S}^{S+1} \quad \forall k \in [K] \\
& & \left.
\begin{aligned}
& s_k + \eta \leq \omega, \ \sqrt{4r_k^2 + (s_k +  \eta)^2} \leq 2\omega - s_k -\eta \\
& \alpha_k + \tr(\bm Q_k\hat{\bm\Omega}_k) + \tr(\bm B_k\hat{\bm\Omega}_k) + \epsilon_k\|\bm Q_k^\top +  \bm B_k\|_F \leq s_k \\
& \pi^k_\ell + \kappa_\ell \theta \geq 0 \quad \forall \ell \in [L+2] \\
& \bm \Delta^k_\ell(\bm x,\bm Y_k,\bm Q_k) - \pi^k_\ell \mathbf{e}_{S+1} \mathbf{e}_{S+1}^\top 
\in \mathcal{C} (\mathcal K_k) \quad \forall \ell \in [L+2]  \\
& \bm B_k + \alpha_k \mathbf{e}_{S+1} \mathbf{e}_{S+1}^\top \in \mathcal{C} (\mathcal K_k)
\end{aligned} \right \} \forall k \in [K],
\end{array}
\end{equation}
where the affine functions $\bm \Delta^k_\ell(\bm x,\bm Y_k,\bm Q_k), \                                     \ell \in [L+2], \ k \in [K]$, are defined in \eqref{eq:pldr_omega}.
\end{theorem}

\begin{proof}
Replacing the semi-infinite constraints in $\eqref{eq:pdr-ref-1}$ with the equivalent copositive reformulations from Proposition \ref{prop:constraint_reformulation}, we arrive at the final COP \eqref{eq:pdr_ref}. This completes the proof. 
\end{proof}

\section{Out-of-Sample Performance Guarantees}
\label{performance_guarantee}
In this section, we establish that with appropriate choices of parameters  carefully chosen ambiguity set size, we can theoretically ensure an attractive finite sample guarantee for the solution of the PDR problem \eqref{eq:pdr}. We first state the following lemma that provides  a high confidence theoretical upper bound on the distance between the second-order moment matrix $\EE_{\PP_k} [\bm\xi\bm\xi^\top]$ and its empirical estimate $\EE_{{\hat{\PP}}_k} [\bm\xi\bm\xi^\top]$. 

\begin{lemma} \textup{(\cite[Corollary 5]{second-moment-bound}).} \label{lem4}
Let ${\hat{\PP}}_k$ be the empirical distribution generated from  $|\mathcal{I}_k|$ independent samples of the distribution $\PP_k \in \mathcal{P} (\Xi_k)$. Then, with probability at least $1-\rho_1$ over the choice of the samples, we have
\begin{equation*}
\left\| \EE_{\PP_k} [\bm\xi\bm\xi^\top] - \EE_{{\hat{\PP}}_k} [\bm\xi\bm\xi^\top] \right\|_F\leq \frac{R_k^2}{\sqrt{|\mathcal{I}_k|}}\left(2+\sqrt{2\ln\frac{1}{\rho_1}}\right),
\end{equation*}
where $R_k=\max_{\bm\xi\in\Xi_k}\|\bm\xi - \hat{\bm \xi} \|_2$. 
\end{lemma}

Furthermore, the following lemma provides a high confidence theoretical upper bound on the $\chi^2$-distance between the partition probability vector $\bm p$ and the empirical estimate $\hat{\bm p}$.


\begin{lemma} \label{lem5}
Let the probability vector $\bm p \in \RR^K$ satisfies $p_k = \PP(\bm \xi \in \Xi_k), \ k \in [K]$. Its empirical estimate $\hat{\bm p} \in \RR^K$ is defined as $\hat{p}_k = \tfrac{1}{N}\sum_{i \in [N]}\mathbbm{1}(\hat{\bm \xi}_i \in \Xi_k)$, $k \in [K]$. Then, we have
\begin{equation}
\label{ineq:probability-bound}
\sum_{k = 1}^{K} (p_k - \hat{p}_k)^2/p_k \leq \frac{1}{N}(K - 1 + \sqrt{-(K-1) \ln{\rho_2}} - 2\ln \rho_2)
\end{equation}
with probability at least $1-C\rho_2$, where C is a constant that depends on the underlying distribution.
\end{lemma}

\begin{proof}
It has been shown that $N \sum_{k = 1}^{K} (p_k - \hat{p}_k)^2/p_k$ can be treated as a Pearson's cumulative test statistic, which asymptotically approaches a $\chi ^{2}$ distribution with $d = K - 1$ degrees of freedom. The paper $\cite{delta-convergence}$ provides the theoretical upper tail bound for $\chi^2$ statistics $z$ with $d$ degrees of freedom as
\begin{equation*}
    \textup{Prob} \left( z \geq d + 2\sqrt{dt} + 2t\right ) \leq e^{-t},
\end{equation*}
where $t \geq 0$. By setting $\rho_2 = e^{-t}$, we obtain 
\begin{equation*}
    \textup{Prob} \left( N \sum_{k = 1}^{K} (p_k - \hat{p}_k)^2/p_k \geq K - 1 + 2\sqrt{-(K - 1)\ln{\rho_2}} - 2\ln{\rho_2} \right ) \leq C\rho_2,
\end{equation*}
where $C$ is a constant that depends on the underlying distribution. This yields the result in $\eqref{ineq:probability-bound}$.
\end{proof}

Based on these theoretical bounds, we can establish that, for carefully chosen sizes of the partition ambiguity sets~$\mathcal{P}_k$, $k \in [K]$, and the uncertainty set $\Delta$, the true underlying distribution $\PP^{\star}$ is contained in the ambiguity set $\mathcal{P}$ with a high probability. This enables us to derive the out-of-sample performance guarantee on the data-driven solution $\hat{\bm x}$ of \eqref{eq:pdr}. 

\begin{theorem} \textup{(Finite sample guarantee).} \label{finite sample guarantee}
Let the ambiguity set $\mathcal{P}$ in \eqref{eq:pdr} be defined as in \eqref{eq:ambiguity_set}. Setting the robustness parameters to $\epsilon_k = \tfrac{R_k^2}{\sqrt{|\mathcal{I}_k|}} \left(2+\sqrt{2\ln\tfrac{K}{\rho_1}}\right)$, $k \in [K]$, and $\gamma = \tfrac{1}{N}\left (K - 1 + 2\sqrt{-(K-1)\ln{\rho_2}} - 2\ln{\rho_2}\right )$, we can ascertain that
\begin{equation*}
    \textup{Prob} \left( J^{\textup{PDR}}  \geq \bm c^\top \hat{\bm x} +  \PP^\star \textup{-CVaR}_\delta [Z(\hat{\bm x}, \bm \xi)] \right) \geq 1  - \rho_1 - C\rho_2,
\end{equation*}
where $\hat{\bm x}$ is an optimal solution of \eqref{eq:pdr}, C is a constant, and $\PP^\star$ is the true underlying distribution.

\end{theorem}

\begin{proof}
According to Lemma \ref{lem4}, by setting 
$$\epsilon_k = \frac{R_k^2}{\sqrt{|\mathcal{I}_k|}} \left(2+\sqrt{2\ln\frac{K}{\rho_1}}\right)$$
in the partition ambiguity set \eqref{eq:ambiguity_set-k}, we have
$$\textup{Prob} \left( \PP_k^\star \notin \mathcal{P}_k \right) \leq \frac{\rho_1}{K} \quad \forall k \in [K].$$
Based on Lemma \ref{lem5}, describing the uncertainty set $\eqref{eq:ambiguity_set-prob}$ with the robustness parameter 
$$\gamma = \frac{1}{N}\left (K - 1 + 2\sqrt{-(K-1)\ln{\rho_2}} - 2\ln{\rho_2}\right),$$
we can verify that
$$\textup{Prob} \left( \bm p^\star \notin \Delta \right) \leq C\rho_2.$$
Using union bound (Boole's inequality), we have 
\begin{equation*}
\begin{array}{cl}
  \displaystyle  \textup{Prob} \left(\{\exists k \in [K] \;\textup{s.t.}\; \PP_k^\star \notin \mathcal{P}_k\}\cup \{\bm p^\star \notin \Delta \}\right) 
    &\displaystyle\leq \sum_{k \in [K]} \textup{Prob} \left( \PP_k^\star \notin \mathcal{P}_k \right) + \textup{Prob} \left( \bm p^\star \notin \Delta \right) \\
    & \displaystyle\leq \rho_1 + C\rho_2,
\end{array}
\end{equation*}
which implies
\begin{equation*}
\begin{array}{cc}
    & \textup{Prob} \left (\bm p^\star \in \Delta, \; \PP_k^\star \in \mathcal{P}_k \ \forall k \in [K] \right) = \textup{Prob} \left(\PP^\star \in \mathcal{P} \right) \geq 1 - \rho_1 - C\rho_2.
\end{array}
\end{equation*}
Accordingly, with $\hat{\theta}$ denoting the optimal solution of \eqref{eq:pdr}, we can ensure with probability at least $1 - \rho_1 - C\rho_2$,
\begin{equation*}
\begin{array}{rlll}
  J^{\textup{PDR}}\geq & \displaystyle \bm c^\top \hat{\bm x} + \hat \theta + \tfrac{1}{\delta} \sum_{k \in [K]}  \PP^\star (\bm \xi \in \Xi_k)\EE_{\PP^\star} [\bm \xi^\top \bm Q_k \bm\xi | \bm\xi\in\Xi_k]  \\
  \geq & \displaystyle \bm c^\top \hat{\bm x} + \hat \theta + \tfrac{1}{\delta} \EE_{\PP^\star} [\max \{ Z(\hat{\bm x},\bm \xi)-\hat{\theta},0 \}]  \\
  \geq & \displaystyle \bm c^\top \hat{\bm x} + \inf_{\theta} \left( \theta + \tfrac{1}{\delta} \EE_{\PP^\star} [\max \{ Z(\hat{\bm x},\bm \xi)-\theta,0 \}] \right)  \\
  = & \bm c^\top \hat{\bm x} + \PP^\star\textup{-CVaR} [Z(\hat{\bm x},  \bm \xi)].
\end{array}
\end{equation*}
The first inequality holds because the ambiguity set $\mathcal{P}$ defined above is verified to contain the true probability~$\PP^{\star}$ with a probability of at least $1-\rho_1 - C\rho_2$. With $\PP^{\star}$ contained in the ambiguity set, the worst-case expectation must be greater than or equal to expectation under the true distribution. The second inequality is satisfied because decision rules provide a conservative approximation. 
The third inequality holds because~$\hat{\theta}$ may not be the optimal value of $\displaystyle \inf_{\theta} \left( \theta + \tfrac{1}{\delta} \EE_{\PP^\star} [\max \{ Z(\hat{\bm x},\bm \xi)-\theta,0 \}] \right)$. The last equality is obtained based on the definition of CVaR. This completes the proof. 
\end{proof}

Theorem \ref{finite sample guarantee} demonstrates that with judicious choices for the ambiguity set parameters, the optimal value~$J^{\textup{PDR}}$ of the PDR problem, which can be solved via COP~$\eqref{eq:pdr_ref}$, provides a $(1 - \rho_1 - C\rho_2)$ confidence bound on the out-of-sample performance of the data-driven solution. Additionally, the above theorem provides guidance for choosing the values of the robustness parameters in practice.

\section{A Decomposition Algorithm}
\label{decomposition_algorithm}
 The copositive program \eqref{eq:pdr_ref} involves $O(K)$ copositive constraints, which introduces additional computational challenges when we improve the approximation quality by increasing the number of partitions. 
 Fortunately, the structure of problem~\eqref{eq:pdr_ref} enables us to solve the subproblems corresponding to different partitions in parallel. 
Exploiting this decomposition structure, in this section we develop an iterative Benders-type algorithm 
 to solve the PDR problem \eqref{eq:pdr_ref}.

To design the decomposition algorithm for  problem~\eqref{eq:pdr_ref}, we consider the following equivalent problem: 
\begin{equation}
\label{eq:pdr_ref2}
\begin{array}{ccll}
J^{\textup{PDR}} = &  \displaystyle \inf & \bm c^\top \bm x + \theta + \frac{1}{\delta} ( \gamma \omega - \eta - 2\hat{\bm p}^\top \bm r + 2\omega\hat{\bm p}^\top \mathbf{e}) \\
& \st & \bm x \in \X, \  \theta, \eta \in \RR,\ \omega \in \RR_{+},\ \bm r, \bm s \in \RR^{K} \\
& & \left.
\begin{aligned}
& s_k + \eta \leq \omega, \
\sqrt{4r_k^2 + (s_k +  \eta)^2} \leq 2\omega - s_k -\eta\\
& Z^{\textup{PDR}}_k \left(\bm x, \theta \right) \leq s_k \\
\end{aligned} \right \} \forall k \in [K].
\end{array}
\end{equation}
Here, the subproblem corresponding to the $k^{\textup{th}}$ partition  is given by
\begin{equation}
\label{eq:pdr_ref2_z}
\begin{array}{ccll}
Z^{\textup{PDR}}_k \left(\bm x, \theta \right) \coloneqq & \displaystyle \inf & \alpha_k + \tr(\bm Q_k\hat{\bm\Omega}_k) + \tr(\bm B_k\hat{\bm\Omega}_k) + \epsilon_k\|\bm Q_k^\top + \bm B_k\|_F  \\
& \st & \bm Y_k \in \RR^{N_2 \times (S+1)}, \ \bm Q_k \in \RR^{(S+1) \times (S+1)}, \ \bm B_k \in \mathbb{S}^{S+1},\ \alpha_k \in \RR, \ \pi_k \in \RR^{L+2} \\
& & \pi^k_\ell + \kappa_\ell \theta \geq 0 \quad \forall \ell \in [L+2] \\
& & \bm \Delta^k_\ell(\bm x,\bm Y_k,\bm Q_k) - \pi^k_\ell \mathbf{e}_{S+1} \mathbf{e}_{S+1}^\top 
\in \mathcal{C} (\mathcal K_k) \quad \forall \ell \in [L+2]  \\
& & \bm B_k + \alpha_k \mathbf{e}_{S+1} \mathbf{e}_{S+1}^\top \in \mathcal{C} (\mathcal K_k),
\end{array}
\end{equation}
where the affine functions $\bm \Delta^k_\ell(\bm x,\bm Y_k,\bm Q_k), \ell \in [L+2], \ k \in [K]$, are defined in \eqref{eq:pldr_omega}. By the equivalence, we can obtain the optimal value and solutions of problem $\eqref{eq:pdr_ref}$ by solving the COP $\eqref{eq:pdr_ref2}$. Moreover, the special block structure of problem $\eqref{eq:pdr_ref2}$ enables us to decompose it into $K$ smaller subproblems $\eqref{eq:pdr_ref2_z}$. 

Before describing the decomposition algorithm, we first introduce the following terminology that will be used to obtain 
the theoretical  results for the algorithm. 
\begin{defi} 
\textup{(Complete recourse under linear decision rules).} We say that the  two-stage distributionally robust optimization problem $\eqref{eq:2s-dro1}$ has complete recourse under linear decision rules if 
there exists $\bm Y \in \RR^{N_2 \times (S+1)}$ such that $(\bm W_\ell\bm \xi)^\top \bm Y \bm \xi > 0$ for all $\bm \xi \in \Xi$ and $\ell \in [L]$.
\end{defi}
The above condition implies the second-stage problem \eqref{eq:second_stage_value}  is always feasible under linear decision rules, i.e., there exists  $\bm Y \in \RR^{N_2 \times (S+1)}$ such that  $\bm y = \bm Y \bm \xi$ is feasible to  \eqref{eq:second_stage_value} for every $\bm x \in \mathcal{X}$ and $\bm \xi \in \Xi$. The existence of feasible LDR also implies the existence of feasible PDR. By the equivalence between problems $\eqref{eq:pdr}$ and $\eqref{eq:pdr_ref2}$, we conclude that the subproblem \eqref{eq:pdr_ref2_z} is feasible for any fixed $\bm x \in \mathcal{X}$ and $\theta \in \RR$, i.e., $Z^{\textup{PDR}}_k \left(\bm x, \theta \right)<~\infty, \ k \in~[K]$.

\begin{lemma}
The complete recourse under linear decision rules is satisfied if and only if there exist \linebreak $\bm Y \in \RR^{N_2 \times (S+1)}$ and $\beta_\ell \in \RR_{++}$, $\ell\in[L]$, such that 
\begin{equation}
\label{eq: complete_recourse_LDR}
    \frac{1}{2}\left (\bm Y^\top \bm W_\ell + \bm W_\ell^\top \bm Y \right ) - \beta_\ell \mathbf{e}_{S+1} \mathbf{e}_{S+1}^\top \in {\mathcal{C}(\mathcal{K})} \quad \forall  \ell \in [L].
\end{equation}
\end{lemma}

\begin{proof}
See the Appendix A.2.
\end{proof}
The condition $\eqref{eq: complete_recourse_LDR}$ can be sufficiently checked by replacing the copositive cone with its inner semidefinite approximations described in Appendix B. 
We also impose the following condition for the theoretical convergence guarantee of the algorithm. 

\begin{defi}
\textup{(Sufficiently expensive recourse)}. The two-stage distributionally robust optimization problem~\eqref{eq:2s-dro1} has sufficiently expensive recourse if for any fixed $\bm x \in \mathcal{X}$ and $\bm \xi \in \Xi$, the dual of the recourse problem~\eqref{eq:second_stage_value} is feasible. 
\end{defi}
\noindent 
The sufficiently expensive recourse condition guarantees that for any fixed $\bm x \in \mathcal{X}$ and $\bm \xi \in \Xi$, the subproblem~
$\eqref{eq:pdr_ref2_z}$ is not unbounded, i.e., $Z^{\textup{PDR}}_k \left(\bm x, \theta \right) > -\infty$, for every $k \in [K]$.

In the following, we show the formulations of the \textit{subproblems} and the \textit{master problem}, and describe the framework of the decomposition algorithm. We also establish that under certain assumptions, the algorithm is guaranteed to terminate in a finite number of iterations.

For fixed first-stage decision variables $(\hat{\bm x}, \hat{\theta})$, we generate the $k^{\textup{th}}$ \textit{subproblem}, which corresponds to the dual problem of the COP $\eqref{eq:pdr_ref2_z}$, as follows:

\begin{equation}
\label{eq:benders_sp}
\begin{array}{clll}
{Z^{\textup{PDR}}_k}^* \left(\hat{\bm x}, \hat{\theta} \right) \coloneqq & \displaystyle \sup & \displaystyle \frac{1}{2}\sum_\ell  \tr\left (\bm H_\ell^k(\mathbf{e}_{S+1} \bm{\mathcal{T}}_\ell (\hat{\bm x})^\top + \bm{\mathcal{T}}_\ell (\hat{\bm x}) \mathbf{e}_{S+1}^\top)\right )  - \sum_\ell \kappa_\ell \hat{\theta} \mathbf{e}_{S+1}^\top \bm H_\ell^k \mathbf{e}_{S+1} \\
& \st & \bm G_k \in \RR^{(S+1) \times (S+1)}, \ \bm O_k \in \mathbb{S}^{S+1}, \ \bm H_\ell^k \in \mathbb{S}^{S+1} \quad \forall \ell \in [L+2] \\
& & \|\bm G_k\|_F \leq \epsilon_k, \ \sum_\ell \bm{\mathcal W}_\ell \bm H_\ell^k = \bm 0 \\
& & \mathbf{e}_{S+1}^\top \bm O_k \mathbf{e}_{S+1} = 1, \ \mathbf{e}_{S+1}^\top \bm H_\ell^k \mathbf{e}_{S+1} \geq 0 \quad \forall \ell \in [L+2] \\
& & \hat{\bm \Omega}_k + \bm G_k - \bm O_k  = \bm 0, \ \hat{\bm \Omega}_k + \bm G_k - \sum_\ell \lambda_\ell \bm H_\ell^k  = \bm 0 \\
& & \bm O_k \in \mathcal{C^*}(\mathcal{K}_k), \ \bm H_\ell^k \in \mathcal{C^*}(\mathcal{K}_k) \quad \forall \ell \in [L+2].\\
\end{array}
\end{equation}

\begin{proposition} \label{prop5}
Suppose the complete recourse under linear decision rules assumption holds, then strong duality holds between the copositive program $\eqref{eq:pdr_ref2_z}$ and the completely positive program $\eqref{eq:benders_sp}$, i.e., $Z^{\textup{PDR}}_k \left(\hat{\bm x}, \hat{\theta} \right) = {Z^{\textup{PDR}}_k}^* \left(\hat{\bm x}, \hat{\theta} \right)$. 
\end{proposition}

\begin{proof}
See the Appendix A.3.
\end{proof}
\noindent In the subproblem $\eqref{eq:benders_sp}$, the first-stage decision variables $\hat{\bm x}$ and $\hat{\theta}$ only appear in the objective, implying that the feasible region of each subproblem is independent of  first-stage decisions. 

At every iteration, the first-stage decision variables $(\hat{\bm x}, \hat{\theta})$ are solved via the \textit{master problem} given by 
\begin{subequations}
\label{eq:benders_mp}
\begin{align}
\underline{J}^{\textup{PDR}} \coloneqq \displaystyle \inf \ & \bm c^\top \bm x + \theta + \frac{1}{\delta} ( \gamma \omega - \eta - 2\hat{\bm p}^\top \bm r + 2\omega\hat{\bm p}^\top \mathbf{e}) \nonumber\\
 \st \ & \bm x \in \X, \  \theta, \eta \in \RR,\ \omega \in \RR_{+},\ \bm r, \bm s \in \RR^{K} \nonumber \\
 & s_k + \eta \leq \omega, \
 \sqrt{4r_k^2 + (s_k +  \eta)^2} \leq 2\omega - s_k -\eta \quad \forall k \in [K] \nonumber\\
 & \displaystyle \frac{1}{2}\sum_\ell  \tr \left (\bm H_\ell^k(\mathbf{e}_{S+1} \bm{\mathcal{T}}_\ell (\bm x)^\top + \bm{\mathcal{T}}_\ell (\bm x) \mathbf{e}_{S+1}^\top)\right ) \nonumber \\[-8pt]
 & \qquad -\sum_\ell \kappa_\ell \theta \mathbf{e}_{S+1}^\top \bm  H_\ell^k \mathbf{e}_{S+1} \leq s_k \quad \forall (\bm H_1^k, \dots, \bm H_L^k ) \in \mathcal{V}^k \quad \forall k \in [K] \label{eq:benders_mp_cons1}\\
 & \displaystyle \frac{1}{2}\sum_\ell  \tr\left (\bm H_\ell^k(\mathbf{e}_{S+1} \bm{\mathcal{T}}_\ell (\bm x)^\top + \bm{\mathcal{T}}_\ell (\bm x) \mathbf{e}_{S+1}^\top)\right )  \nonumber \\[-8pt]
 & \qquad-\sum_\ell \kappa_\ell \theta \mathbf{e}_{S+1}^\top \bm H_\ell^k \mathbf{e}_{S+1} \leq 0 \quad \forall (\bm H_1^k, \dots, \bm H_L^k ) \in \mathcal{W}^k \quad \forall k \in [K], \label{eq:benders_mp_cons2}
\end{align}
\end{subequations}
which  involves tractable linear and second-order cone constraints only. 

The  constraint systems \eqref{eq:benders_mp_cons1} and \eqref{eq:benders_mp_cons2} serve as the \textit{optimality cuts} and the \textit{feasibility cuts}, respectively.  The left-hand sides of the constraints correspond to the objective function of the subproblem \eqref{eq:benders_sp}.  
To generate an \textit{optimality cut}  when the subproblem is feasible but not optimal, i.e., $\hat{s}_k \leq {Z^{\textup{PDR}}_k}^* \left(\hat{\bm x}, \hat{\theta} \right) < \infty$, we add the solution $(\bm H_\ell^k)_{\ell \in [L]}$ to the set $\mathcal{V}^k$ by solving the subproblem~$\eqref{eq:benders_sp}$. 
On the other hand, to generate a \textit{feasibility cut} in the master problem when the subproblem $\eqref{eq:benders_sp}$ is unbounded, i.e., ${Z^{\textup{PDR}}_k}^* \left(\hat{\bm x}, \hat{\theta} \right) = \infty$, we add the solution $(\bm H_\ell^k)_{\ell \in [L]}$ to the set $\mathcal{W}^k$ by solving
\begin{equation}
\label{eq:benders_fea_cut}
\begin{array}{ccl}
(\bm H_\ell^k)_{\ell \in [L]} \in & \displaystyle \textup{argmax}  & \displaystyle \frac{1}{2}\sum_\ell \left( \tr(\bm H_\ell^k(\mathbf{e}_{S+1} \bm{\mathcal{T}}_\ell (\hat{\bm x})^\top + \bm{\mathcal{T}}_\ell (\hat{\bm x}) \mathbf{e}_{S+1}^\top)) \right) - \sum_\ell \kappa_\ell \hat{\theta} \mathbf{e}_{S+1}^\top \bm H_\ell^k \mathbf{e}_{S+1} \\
& \st & \bm G_k \in \RR^{(S+1) \times (S+1)}, \ \bm O_k \in \mathbb{S}^{S+1}, \ \bm H_\ell^k \in \mathbb{S}^{S+1} \quad \forall \ell \in [L+2] \\
& & \|\bm G_k\|_F \leq \epsilon_k, \ \sum_\ell \bm{\mathcal W}_\ell \bm H_\ell^k = \bm 0 \\
& & \mathbf{e}_{S+1}^\top \bm O_k \mathbf{e}_{S+1} = 0, \ \mathbf{e}_{S+1}^\top \bm H_\ell^k \mathbf{e}_{S+1} \geq 0 \quad \forall \ell \in [L+2] \\
& & \bm G_k - \bm O_k  = \bm 0, \ \bm G_k - \sum_\ell \lambda_\ell \bm H_\ell^k  = \bm 0\\
& & \bm O_k \in \mathcal{C^*}(\mathcal{K}_k), \ \bm H_\ell^k \in \mathcal{C^*}(\mathcal{K}_k) \quad \forall \ell \in [L+2] \\
& & -\mathbf{1} \leq \bm G_k \leq \mathbf{1},\ -\mathbf{1} \leq \bm O_k \leq \mathbf{1}, \ -\mathbf{1} \leq \bm H_\ell^k \leq \mathbf{1} \quad \forall \ell \in [L+2]. \\
\end{array}
\end{equation}
\noindent Problem \eqref{eq:benders_fea_cut} finds the direction in which the subproblem $\eqref{eq:benders_sp}$ becomes unbounded as the decision variables keep increasing. The decomposition algorithm consecutively adds the optimal solution $(\bm H_\ell^k)_{\ell \in [L]}$ obtained from solving problems \eqref{eq:benders_sp} and \eqref{eq:benders_fea_cut} to the sets $\mathcal{V}^k$ and $\mathcal{W}^k$ at each iteration. 

Based on the above discussions, we describe the generalized decomposition procedure in Algorithm~\ref{alg:decomposition_algorithm}.
\begin{algorithm}[h] 
 \setstretch{0.9}
 \caption{The Decomposition Algorithm for Two-stage DRO}
 \begin{algorithmic}[1]  
 \label{alg:decomposition_algorithm}
 \STATE \textbf{Input:} parameters of problem \eqref{eq:pdr}, tolerance level $\eta \geq 0$
 \STATE \textbf{Output:} solution $(\bm x, \theta)$ to problem \eqref{eq:pdr} within $100\eta\%$ of optimal
 \STATE INITIALIZE $\overline{J}^{\textup{PDR}} = \infty, \mathcal{V}^k = \mathcal{W}^k = \emptyset \quad \forall k \in [K]$
 \WHILE{true}
    \STATE \textbf{Step 1} \\
     Solve the master problem \eqref{eq:benders_mp} and obtain the optimal value $\underline{J}^{\textup{PDR}}$ \\
    \STATE \textbf{Step 2} \\
    \FOR{$k \in [K]$}
    \STATE Given fixed $(\hat{\bm x}, \hat{\theta}, \left \{ \hat{s}_k \right \}_{k \in [K]})$, solve the subproblem \eqref{eq:benders_sp} and get the optimal value ${Z^{\textup{PDR}}_k}^* \left(\hat{\bm x}, \hat{\theta} \right)$
    \ENDFOR \\
    Let $\hat{J}^{\textup{PDR}} = \bm c^\top \hat{\bm x} + \hat{\theta} + \frac{1}{\delta} \sum_{k\in[K]} {Z^{\textup{PDR}}_k}^* \left(\hat{\bm x}, \hat{\theta} \right)$ \\
        \IF{$\hat{J}^{\textup{PDR}} < \overline{J}^{\textup{PDR}} $}
        \STATE Set $\overline{J}^{\textup{PDR}} = \hat{J}^{\textup{PDR}}$, and $\left ( \bm x^\star, \theta^\star \right) = \left(\hat{\bm x}, \hat{\theta} \right)$
        \ENDIF
    \STATE \textbf{Step 3} 
    \IF{$\overline{J}^{\textup{PDR}} - \underline{J}^{\textup{PDR}} \leq \eta \min \left ( |\overline{J}^{\textup{PDR}}|,|\underline{J}^{\textup{PDR}}| \right )$}
    \STATE Stop and output $\left ( \bm x^\star, \theta^\star \right )$
    \ENDIF
    \STATE \textbf{Step 4}
    \FOR{$k \in [K]$}
    \IF{${Z^{\textup{PDR}}_k}^* \left(\hat{\bm x}, \hat{\theta} \right) = \infty$}
    \STATE Add the feasibility cut by solving problem \eqref{eq:benders_fea_cut} and updating $\mathcal{W}_k = \mathcal{W}_k \bigcup ( \bm H_1^k, \dots, \bm H_L^k )$
    \ELSIF{$\hat{s}_k < {Z^{\textup{PDR}}_k}^* \left(\hat{\bm x}, \hat{\theta} \right) < \infty$}
    \STATE Add the optimality cut by solving problem \eqref{eq:benders_sp} and updating $\mathcal{V}_k = \mathcal{V}_k \bigcup ( \bm H_1^k, \dots, \bm H_L^k )$
    \ENDIF
    \ENDFOR
 \ENDWHILE
 \end{algorithmic}
 \end{algorithm}
 The algorithm guarantees finite convergence, as follows.
\begin{theorem} \textup{(Finite $\eta$-convergence).} \label{finite convoergence}
Assume $\mathcal{X}$ is a nonempty compact convex set or a finite discrete set. When the two-stage DRO problem $\eqref{eq:2s-dro1}$ satisfies both the complete recourse under linear decision rules and sufficiently expensive recourse assumptions, Algorithm~\ref{alg:decomposition_algorithm} converges in a finite number of steps for any given tolerance level $\eta > 0$.
\end{theorem}

\begin{proof}
The decomposition algorithm is analogous to the generalized Benders decomposition algorithm \cite{GBD1}. 
When both complete recourse under linear decision rules and sufficiently expensive recourse assumptions are satisfied,  the subproblem \eqref{eq:benders_sp} for each partition is feasible and bounded for any fixed first-stage decisions $(\hat{\bm x}, \hat{\theta})$, i.e., $-\infty < {Z^{\textup{PDR}}_k}^* \left(\hat{\bm x}, \hat{\theta} \right)< \infty$. Furthermore, according to the proof of Proposition \ref{prop5}, there exists a strictly feasible solution (or a Slater point) to each subproblem \eqref{eq:pdr_ref2_z}. 
Then, based on the finite $\eta$-convergence of the generalized Benders decomposition (\cite[Theorems 2.4 and 2.5]{GBD1}), we can conclude that the PDR problem $\eqref{eq:pdr_ref}$ is guaranteed to converge to the global optimum for any given tolerance level $\eta > 0$ in a finite number of iterations.
\end{proof}
For the special case when the first-stage decision $\bm x$ is integer, the benefit of the decomposition algorithm is significant because it breaks up the complexity of the original problem, which constitutes a mixed-integer COP. 
Indeed, the master problem is now a more tractable mixed-integer second-order cone program, for which off-the-shelf solvers are applicable. 
Furthermore, as the feasible set of the first-stage decision variables is finite,  the decomposition algorithm solves the problem to optimality in a finite number of iterations (\cite[Theorem 2.4]{GBD1}). 

We remark that generalized copositive programs and generalized completely positive programs are intractable \cite{copositive-program}. To develop an efficient solution for the copositive programming problem \eqref{eq:pdr_ref}, we replace the copositive cones $\mathcal{C}(\mathcal K_k), \ k \in [K]$, with semidefinite-representable inner approximations $\mathcal{IA}_0(\mathcal K_k)$ and $\mathcal{IA}_1(\mathcal K_k)$ presented in \eqref{eq:sdp1} and \eqref{eq:sdp2} in Appendix~B. In this way, we obtain conservative solutions by solving tractable semidefinite programs. Similarly, we derive the corresponding semidefinite approximations \eqref{eq:sdp1-dual} and~\eqref{eq:sdp2-dual} in Appendix B for the completely positive programs \eqref{eq:benders_sp} and \eqref{eq:benders_fea_cut} in the decomposition algorithm. In Appendix B, we further provide detailed comparisons of different semidefinite approximations for problems \eqref{eq:pdr_ref}, \eqref{eq:benders_sp}, and $\eqref{eq:benders_fea_cut}$, and the corresponding optimal objective values of the semidefinite reformulations.

\section{Numerical Experiments}
\label{experiments}

In this section, we demonstrate the effectiveness of our copositive programming approach to two-stage DRO problems over three applications. The first experiment is on  network inventory allocation problem, which uses expectation as the risk measure and has random recourse in the objective. The second experiment, multi-item newsvendor, and the third experiment, medical scheduling, have CVaR as risk measure, where their expectation reformulations contain random recourse in the constraints. In Appendix~C, we conduct another experiment on facility location problem, which is a two-stage optimization problem with binary first-stage decision variable and random recourse in both the objective and the constraints. 

We compare different methods by evaluating their 
out-of-sample performances and computation times. The various data-driven methods include:

$\bm \cdot$ \textbf{$\mathbf{C_0}$ SDP} - The semidefinite approximation using the cone $\mathcal{IA}_0\mathcal{(K)}$ in  \eqref{eq:sdp1} for the copositive reformulation \eqref{eq:pdr_ref}, where the ambiguity set is defined in $\eqref{eq:ambiguity_set}$ with two layers of robustness. 

$\bm \cdot$ \textbf{$\mathbf{C_1}$ SDP} - The semidefinite approximation using the cone $\mathcal{IA}_1\mathcal{(K)}$   in \eqref{eq:sdp2}  for the copositive reformulation \eqref{eq:pdr_ref}, where the ambiguity set is defined in $\eqref{eq:ambiguity_set}$ with two layers of robustness.

$\bm \cdot$ \textbf{Wass SDP} - The semidefinite approximation for the copositive reformulation of two-stage DRO with the Wasserstein ambiguity set, proposed by Hanasusanto and Kuhn \cite{Wass-SDP}. We only apply this method in the first experiment because this scheme is only applicable for solving two-stage DRO problems with random recourse in the objective.

$\bm \cdot$ \textbf{SAA} - The sample average approximation.\\
All optimization problems are solved using MOSEK 9.2.28 via the YALMIP interface on a 10-core 2.8GHz to 5.2 GHz Windows PC with 32 GB RAM.

\subsection{Network inventory allocation}
\subsubsection{Problem Description} 
We consider a two-stage capacitated network inventory problem with $M$ locations where one determines the stock allocations $x_i,\ i \in [M]$, before knowing the realizations of demands $u_i$, $i\in [M]$, and per-unit transportation costs $v_{ij}$ from location $i$ to $j$, $i,j\in[M]$. The unit cost of buying stock in advance at location~$i$ is $c_i$. The demand $u_i$ in location $i$ can be satisfied by existing stock $x_i$ and by transporting $y_{ji}$ units from location $j$. The stock capacity at each location is $T$ units. The network inventory allocation problem minimizes the worst-case storage  and transportation cost given by
\begin{equation*}
\begin{array}{clll}
  \displaystyle  \text{min} & \displaystyle \bm c^\top \bm x + \sup_{\PP\in\mathcal P}\EE_\PP [Z(\bm x, \bm u, \bm v)] \\
  \st & \bm x \in \mathbb{R^M}, \ 0 \leq x_i \leq T \quad \forall  i \in [M], \\ 
\end{array}
\end{equation*}
where the second-stage cost is given by
\begin{equation*}
\begin{array}{ccll}
  \displaystyle Z(\bm x, \bm u, \bm v) \coloneqq & \displaystyle \text{min} & \displaystyle \sum_{i \in [M]} \sum_{j \in [M]} v_{ij}y_{ij} \\
  & \st & \bm y \in \mathbb{R_{+}^{M \times M}} \\
   &  & x_i +  \underset{j \in [M]}{\sum} y_{ji} - \underset{j \in [M]}{\sum} y_{ij} \geq u_i \quad \forall  i \in [M]. \\ 
\end{array}
\end{equation*}

In this problem, we assume the only information provided is the historical observations of the random parameters, $(\hat{\bm u}_1, \hat{\bm v}_1),\ldots,(\hat{\bm u}_N, \hat{\bm v}_N)$, and the description of the support set $ \Xi= \{ (\bm u, \bm v) \in (\RR^{M},\RR^{M \times M}) :\underline{u}_l \leq u_i \leq \overline{u}_u, \;\underline{v}_l \leq v_{ij} \leq \overline{v}_u,\ i,j \in [M]\} $.

\subsubsection{Experiments}
Since the problem has random recourse in the objective and fixed recourse in the constraints, we apply PLDR introduced in Section \ref{sec:PLDR} to the second-stage decision variable $\bm y(\bm u, \bm v)$. We further approximate the copositive problem as semidefinite programs using the cone $\mathcal{IA}_0\mathcal{(K)}$ and the cone $\mathcal{IA}_1\mathcal{(K)}$ separately, denoted as \textbf{$\mathbf{C_0}$ SDP} and \textbf{$\mathbf{C_1}$ SDP}. Here we set the number of partitions $K$  to the size of historical dataset $N$. Aside from \textbf{$\mathbf{C_0}$ SDP} and \textbf{$\mathbf{C_1}$ SDP}, we also implement the simplified SDP solution schemes that use empirical partition probabilities by setting the robustness parameter $\gamma$ in the uncertainty set \eqref{eq:ambiguity_set-prob} to $0$, given by \textbf{$\mathbf{\hat{C}_0}$ SDP} and \textbf{$\mathbf{\hat{C}_1}$ SDP}. We compare these SDP solution schemes with \textbf{Wass SDP} and \textbf{SAA}. Furthermore, we apply Algorithm~\ref{alg:decomposition_algorithm} and show the advantages of using the decomposition algorithm to accelerate the solution process.
To calculate the runtime of the decomposition algorithm, we choose the maximum among the runtimes of solving the $K$ subproblems and add that to the total time of running the algorithm. We note, however,  that the decomposition algorithm can naturally be implemented in parallel computing environments. 

We evaluate the performances of the above solution schemes by conducting out-of-sample experiments with $N = 10, 20, 40, 80, 160$ independent training samples and 50,000 independent testing samples. The results of all experiments are averaged over 100 random instances. For each instance, we generate a network of size $M=5$ with stock costs $\bm c = [40,50,60,70,80]$. We set $T = 80, \ \underline{u}_l = 20, \ \overline{u}_u = 40,\  \underline{v}_l =~40, \ \overline{v}_u =~50$. We assume the the true distributions of the demands $\bm u$ are truncated lognormal with means $\bm\mu_1 = [3,3,3.5,3.5,3.5]$ and standard deviations $\bm \sigma_1 = 0.2 \mathbf{e}$. The true distributions of the per-unit transportation costs $\bm v$ are also assumed to be truncated lognormal with  means $\bm\mu_2 = 3 \mathbf{e} $ and standard deviations $\bm \sigma_2 = 0.1\mathbf{e}$. For the robustness parameters $\epsilon_k, \ k \in [K]$, in the ambiguity set, we use 2-fold cross validation to determine the values; for the other robustness parameter $\gamma$, we set it according to the theoretical value where $\rho_2 = 0.1$.

\subsubsection{Results Analysis}

\begin{table}[t!]
\color{black}
\centering
\begin{tabular}{|c|rrrrrr|}\hline
$N$  & \textbf{$\mathbf{C_0}$ SDP} & \textbf{$\mathbf{C_1}$ SDP} & \textbf{$\mathbf{\hat{C}_0}$ SDP} & \textbf{$\mathbf{\hat{C}_1}$ SDP} &  \textbf{Wass SDP}   & \textbf{SAA} \\
[0.0mm] \hline
$10$  & 100\% & 100\% & 100\% & 100\% & 100\% & 0.5\% \\
$20$  & 100\% & 100\% & 100\% & 100\% & 100\% & 0.4\% \\
$40$  & 100\% & 100\% & 100\% & 100\% & 100\% & 0.8\% \\ 
$80$  & 100\% & 100\% & 100\% & 100\% & 100\% & 0.7\% \\
$160$ & 100\% & 100\% & 100\% & 100\% & 100\% & 3.9\% \\
\hline\hline
\end{tabular}
\caption[]{The average percent of realization where the method's first-stage decision is feasible over 10,000 trials.}
\label{tbl:inventory1}
\end{table}

Table~\ref{tbl:inventory1} shows the proportions of realizations where the first-stage decision $\bm x$ is feasible in the out-of-sample tests. 
The results show that as expected all the SDP solution schemes are guaranteed to produce feasible first-stage decisions when evaluating the out-of-sample performance. On the other hand,  SAA exhibits poor feasibility performance. 
As such, we no longer consider its out-of-sample performance and runtime in subsequent comparisons.

\begin{table}[t!]
\color{black}
\centering
\begin{tabular}{|c|rrrrr|}\hline
$N$  & \textbf{$\mathbf{C_0}$ SDP} & \textbf{$\mathbf{C_1}$ SDP}  & \textbf{$\mathbf{\hat{C}_0}$ SDP} & \textbf{$\mathbf{\hat{C}_1}$ SDP} &  \textbf{Wass SDP}  \\
[0.0mm] \hline
$10$  & 11352 & 11352 & 11388 & 11388 & 11516 \\
$20$  & 11325 & 11325 & 11365 & 11365 & $\backslash$ \\
$40$ & 11306 & $\backslash$ & 11340 & $\backslash$ & $\backslash$ \\ 
$80$ & 11297 & $\backslash$ & 11297 & $\backslash$ & $\backslash$ \\
$160$ & 11272 & $\backslash$ & 11248 & $\backslash$ & $\backslash$ \\
\hline\hline
\end{tabular}
\caption[]{The average out-of-sample cost of different approaches over 100 trials.}
\label{tbl:inventory2}
\end{table}

Table~\ref{tbl:inventory2} compares the out-of-sample performances of different approaches. We assume the method is applicable only if its runtime is less than 900 seconds. We observe that the out-of-sample costs of \textbf{$\mathbf{C_0}$ SDP}, \textbf{$\mathbf{C_1}$ SDP}, \textbf{$\mathbf{\hat{C}_0}$ SDP} and \textbf{$\mathbf{\hat{C}_1}$ SDP} decrease as the number of in-sample points increases. Semidefinite approximations using the inner cones $\mathcal{IA}_0\mathcal{(K)}$ and  $\mathcal{IA}_1\mathcal{(K)}$ perform almost the same  for instances with small  in-sample data sizes. Additionally, we observe that restricting the uncertainty set $\Delta$ to contain only the empirical partition probabilities $\hat{\bm p}$ does not sacrifice the performances of \textbf{$\mathbf{\hat{C}_0}$ SDP} and \textbf{$\mathbf{\hat{C}_1}$ SDP}. As such, in the following experiments, we will directly implement the simpler SDP schemes by setting the robustness parameter to $\gamma = 0$ for both approximation schemes. The results in Table~\ref{tbl:inventory2} further show that \textbf{Wass SDP} does not work well for this problem.

\begin{table}[t!]
\color{black}
\centering
\begin{tabular}{|c|rrrrrrr|}\hline
$N$   & \textbf{$\mathbf{C_0}$ SDP} & \textbf{$\mathbf{C_1}$ SDP}  & \textbf{$\mathbf{\hat{C}_0}$ SDP} & \textbf{$\mathbf{\hat{C}_1}$ SDP}  & \textbf{Benders $\mathbf{C_0}$} & \textbf{Benders $\mathbf{\hat{C}_0}$} &  \textbf{Wass SDP} \\
[0.0mm] \hline
$10$  & 3.7275 & 16.4684 & 4.0807 & 13.9272 & 5.2203 & 5.2326 & 737.8000 \\
$20$  & 9.7430 & 76.2147 & 9.3361 & 81.9130 & 6.1783 & 5.7682 & $\backslash$ \\
$40$ & 28.0969 & $\backslash$ & 29.5526 & $\backslash$ & 7.2758 & 7.2570 & $\backslash$ \\ 
$80$ & 97.2413 & $\backslash$ & 67.0260 & $\backslash$  & 10.5836 & 9.9920 & $\backslash$ \\
$160$ & 351.9841 & $\backslash$ & 282.4895 & $\backslash$ & 16.5239 & 15.0714 & $\backslash$ \\
\hline\hline
\end{tabular}
\caption[]{The average runtime of different approaches over 100 trials.
}
\label{tbl:inventory3}
\end{table}

With regard to computational costs,  \textbf{$\mathbf{C_0}$ SDP}, \textbf{$\mathbf{C_1}$ SDP}, \textbf{$\mathbf{\hat{C}_0}$ SDP} and \textbf{$\mathbf{\hat{C}_1}$ SDP} have notably less runtime than \textbf{Wass SDP} as shown in Table~\ref{tbl:inventory3}. One possible reason is that \textbf{Wass SDP}  solves a large SDP problem with $O(N)$ semidefinite constraints, in which each constraint involves $O((S+N_2+L)^2)$ variables. In comparison, \textbf{$\mathbf{C_0}$ SDP}, \textbf{$\mathbf{C_1}$ SDP}, \textbf{$\mathbf{\hat{C}_0}$ SDP} and \textbf{$\mathbf{\hat{C}_1}$ SDP} solve a SDP problem with $O(K)$ semidefinite constraints, where each constraint involves $O(S^2)$ variables only. However, the SDP solution schemes become less time-efficient as the number of partitions increases. It would thus be advantageous to leverage Algorithm~\ref{alg:decomposition_algorithm} which solves the subproblems in parallel. Because \textbf{$\mathbf{C_1}$ SDP} and \textbf{$\mathbf{\hat{C}_1}$ SDP} take much longer time to solve but deliver similar out-of-sample performances to \textbf{$\mathbf{C_0}$ SDP} and \textbf{$\mathbf{\hat{C}_0}$ SDP}, we conclude \textbf{$\mathbf{C_0}$ SDP} and \textbf{$\mathbf{\hat{C}_0}$ SDP} are a better choice when balancing optimality and scalability. Thus, we only apply the decomposition algorithm to solve \textbf{$\mathbf{C_0}$ SDP} and \textbf{$\mathbf{\hat{C}_0}$ SDP}, denoted as \textbf{Benders $\mathbf{C_0}$} and \textbf{Benders $\mathbf{\hat{C}_0}$}. We set the tolerance level $\eta = 0.05$ when solving the decomposition algorithm. Numerically this algorithm significantly reduces runtime compared to the original SDP schemes particularly when the number of in-sample points is large.

\subsection{Multi-item newsvendor}

\subsubsection{Problem Description} 

In this experiment, we explore an inventory management problem where we want to decide the order quantities $x_i$, $i \in [M]$, of $M$ products, before knowing the random demands $\xi_i$, $i \in [M]$, and the random stockout costs $s_i$, $i \in [M]$. There is a limit $B$ of the total order quantity. We incur a per-unit stockout cost $s_i$ when the order quantity $x_i$ is less than the demand $\xi_i$ and a holding cost $g_i$ if the order quantity $x_i$ is more than the demand $\xi_i$. The second-stage variable $y_{1,i}\coloneqq\max\{x_i-\xi_i,0\}$ corresponds to the excess amount, while  $y_{2,i}\coloneqq\max\{\xi_i-x_i,0\}$ corresponds to the shortfall amount. 
Our goal is to minimize the worst-case CVaR of the total stockout and holding costs:
\begin{equation*}
\begin{array}{clll}
  \displaystyle \text{min} & \displaystyle \sup_{\PP\in\mathcal P} \PP\text{-CVaR}_\epsilon[Z(\bm x, \bm \xi, \bm s)] \\
  \st & \bm x \in \mathbb{R}^M_{+}, \ \mathbf{e}^\top \bm x \leq B, \\
\end{array}
\end{equation*}
where the second-stage cost is defined as
\begin{equation*}
\begin{array}{ccll}
  \displaystyle Z(\bm x, \bm \xi, \bm s) \coloneqq & \displaystyle \text{inf} & \displaystyle \bm g^\top \bm y_1  +\bm s^\top \bm y_2 \\
   & \st &\bm y_1 \in \mathbb{R}^M_{+},\ \bm y_2 \in \mathbb{R}^M_{+}\\
   & & \bm y_1 \geq \bm x - \bm \xi, \ \bm y_2 \geq \bm \xi- \bm x.
\end{array}
\end{equation*}
According to the definition of CVaR, it can be shown that this problem is equivalent to
\begin{equation}
\label{exp:newsvendor2}
\begin{array}{clll}
   \displaystyle \text{inf} & \displaystyle \theta + \frac{1}{\delta} \sup_{\PP\in\mathcal P}\EE_\PP [\tau(\bm \xi,\bm s)]  \\
   \st &\theta\in \mathbb{R}, \ \bm x \in \mathbb{R}^M_{+}, \ \bm y_1: \Xi \rightarrow \RR^{M}, \ \bm y_2: \Xi \rightarrow \RR^{M}, \ \tau: \Xi \rightarrow \RR\\
   & \mathbf{e}^\top \bm x \leq B \\
   & \left.\begin{aligned}
   & \tau(\bm \xi, \bm s) \geq 0, \ \tau(\bm \xi, \bm s) \geq \bm g^\top\bm y_1(\bm \xi,\bm s) +\bm s^\top \bm y_2(\bm \xi,\bm s) -  \theta \\
   & \bm y_1(\bm \xi,\bm s) \geq \bm 0, \ \bm y_1(\bm \xi,\bm s) \geq \bm x - \bm \xi \\
   & \bm y_2(\bm \xi,\bm s) \geq \bm 0,\ \bm y_2(\bm \xi,\bm s) \geq \bm \xi- \bm x
\end{aligned}\right\}  \forall(\bm \xi, \bm s) \in \Xi.
\end{array}
\end{equation}
We assume the only information provided is the historical data, $(\hat{\bm \xi}_1, \hat{\bm s}_1),\ldots,(\hat{\bm \xi}_N, \hat{\bm s}_N)$, and the description of the support set  $ \Xi= \{ (\bm \xi, \bm s) \in (\RR^{M},\RR^{M}) :\underline{\xi}_l \leq \xi_i \leq \overline{\xi}_u, \ \underline{s}_l \leq s_{i} \leq \overline{s}_u,\ i \in [M]\}$. 

\subsubsection{Experiments}
The reformulated problem \eqref{exp:newsvendor2} has random recourse in the constraints. After applying PLDR to the second-stage decision variables $\bm y_1(\bm \xi,\bm s)$ and $\bm y_2(\bm \xi,\bm s)$, the right-hand side of the third inequality contains a quadratic term of random parameters $(\bm \xi, \bm s)$. Thus, we apply PQDR to the second-stage decision variable $\tau(\bm \xi, \bm s)$ on the left-hand side. After applying the decision rules, we separately use the inner approximations $\mathcal{IA}_0\mathcal{(K)}$ and $\mathcal{IA}_1\mathcal{(K)}$ proposed in Appendix B and solve the corresponding semidefinite programs denoted as \textbf{$\mathbf{C_0}$ SDP} and \textbf{$\mathbf{C_1}$ SDP}. Here, we further take advantage of the decomposition algorithm to reduce the runtimes of the semidefinite programs. 

In order to assess the performances of \textbf{$\mathbf{C_0}$ SDP}, \textbf{$\mathbf{C_1}$ SDP} and \textbf{SAA}, we compute the out-of-sample performance of the multi-item newsvendor problem  with $M = 5$ products. The training datasets are of sizes $N = 10, 20, 40, 80, 160$, and testing datasets contain 50,000 independent samples. We evaluate the out-of-sample performance by averaging over 100 trials. For each trial, we set $B = 30, \ \underline{\xi}_l = 0, \ \overline{\xi}_u =~10, \ \underline{s}_l =~0, \ \overline{s}_u = 50$ and the holding cost vector $\bm g = [5,6,7,8,9]$. The parameter $\delta$ related to the risk attitude of CVaR is fixed to~0.1. 
The true distributions of the demands $\bm \xi$ are assumed to be truncated lognormal with means $\bm \mu_1 = \mathbf{e}$ and standard deviations $\bm \sigma_1 = \mathbf e$, and the true distributions of the stockout costs $\bm s$ are assumed to be truncated lognormal with means $\bm \mu_2 = 3 \mathbf{e}$ and standard deviations $\bm \sigma_2 = 2 \mathbf{e}$. We determine the values of the robustness parameters~$\epsilon_k$, $k \in [K]$, in the ambiguity set using a 2-fold cross validation procedure. The robustness parameter $\gamma$ for the $\chi^2$ uncertainty set is set to $0$. The tolerance level $\eta$  of the decomposition algorithm is set to  $0.05$.

\subsubsection{Results Analysis}

\begin{figure}[t!]
		\includegraphics[width=0.5\linewidth]{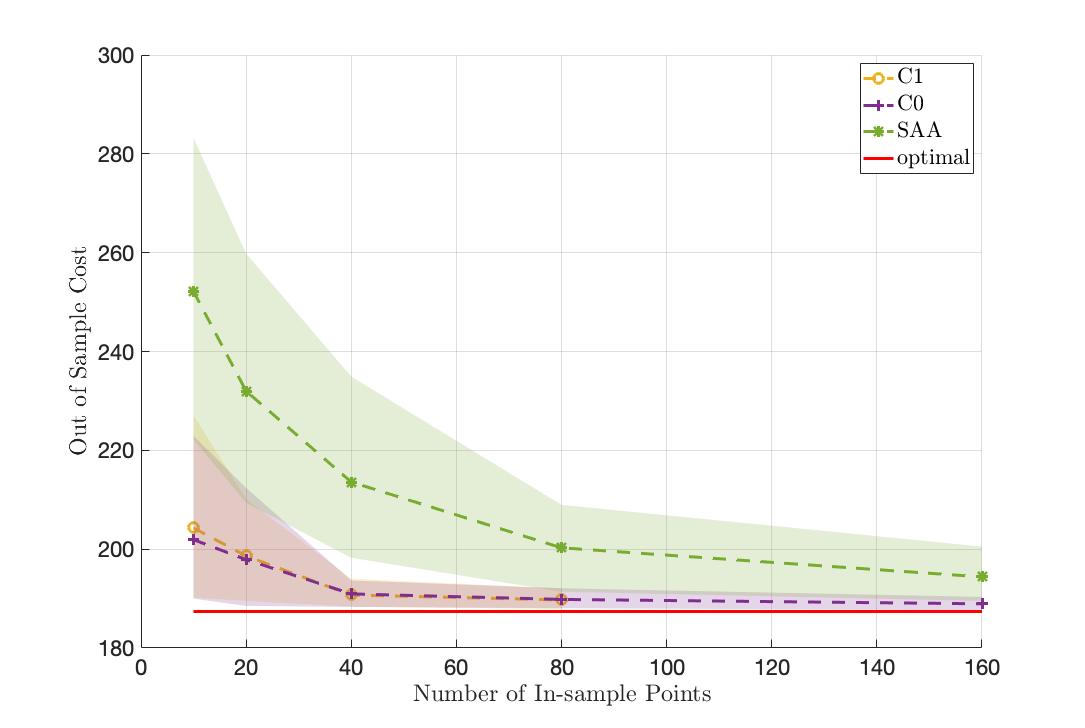}
		\caption{The out-of-sample cost  as a function of the number of in-sample data points. The dashed lines represent the average values of the out-of-sample cost, and the shaded areas visualize the $10\%$ to $90\%$ quantile range over 100 trials.}
		    \label{fig:newsvendor_performance}
\end{figure}
 	
\begin{table}[t!]
\color{black}
\centering
\begin{tabular}{|c|rrrr|}\hline
$N$  & \textbf{$\mathbf{C_1}$ SDP}  & \textbf{$\mathbf{C_0}$ SDP} &  \textbf{Benders $\mathbf{C_0}$} & \textbf{SAA} \\
[0.0mm] \hline
$10$  & 2.0477 & 0.2490 & 0.5769 & 0.0033\\
$20$  & 10.7969 & 0.6176 & 0.6913 & 0.0041\\
$40$ &  62.6185 & 1.9554 & 0.6843 & 0.0068\\ 
$80$ &  549.9155 & 7.7068 & 1.0081 & 0.0135\\
$160$ & $\backslash$ & 41.8532 & 1.7745 & 0.0236\\
\hline\hline
\end{tabular}
\caption[]{The average runtime of different approaches over 100 trials.
}
\label{tbl:newsvendor_runtime}
\end{table}

We compare the out-of-sample performances of \textbf{$\mathbf{C_0}$ SDP}, \textbf{$\mathbf{C_1}$ SDP} and \textbf{SAA} as shown in Figure~\ref{fig:newsvendor_performance}. Here we assume the method is applicable only if its runtime is less than 900 seconds. In Figure~\ref{fig:newsvendor_performance}, we observe that the out-of-sample costs of all three methods decrease with the  training data size and approach the optimal cost. 
Both \textbf{$\mathbf{C_0}$ SDP} and \textbf{$\mathbf{C_1}$ SDP} achieve better out-of-sample performances with approximately $20\%$ improvement compared to \textbf{SAA} when the in-sample data is limited. Comparing the two SDP schemes, we notice that \textbf{$\mathbf{C_0}$ SDP} has a marginal advantage over \textbf{$\mathbf{C_1}$ SDP} in terms of out-of-sample performance.

Table~\ref{tbl:newsvendor_runtime} shows the runtimes of all three methods above, with the addition of \textbf{Benders $\mathbf{C_0}$} which solves \textbf{$\mathbf{C_0}$ SDP} using the decomposition algorithm. Our results confirm the benefit of applying the decomposition algorithm \textbf{Benders $\mathbf{C_0}$}. This advantage grows with the number of in-sample data. 

 	\begin{figure}[t!]
 		\begin{center}
		\begin{subfigure}[b]{.5\textwidth}
			\includegraphics[width=1\linewidth]{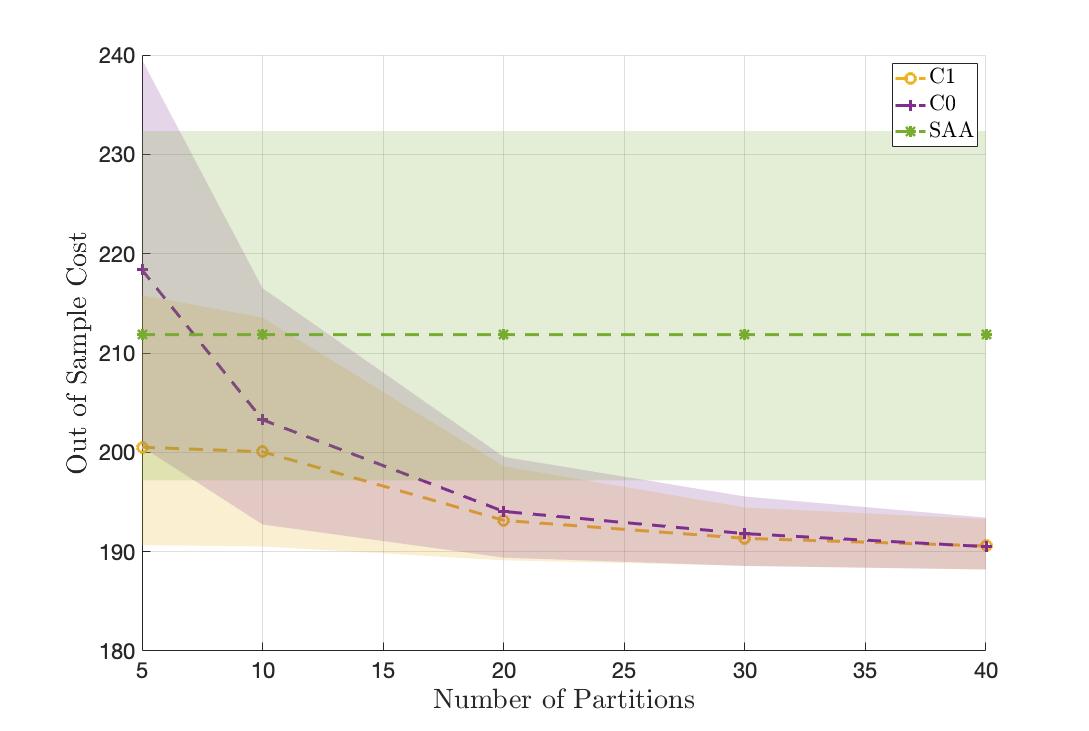}
		\end{subfigure}\hfill%
		\begin{subfigure}[b]{.5\textwidth}
			\includegraphics[width=1\linewidth]{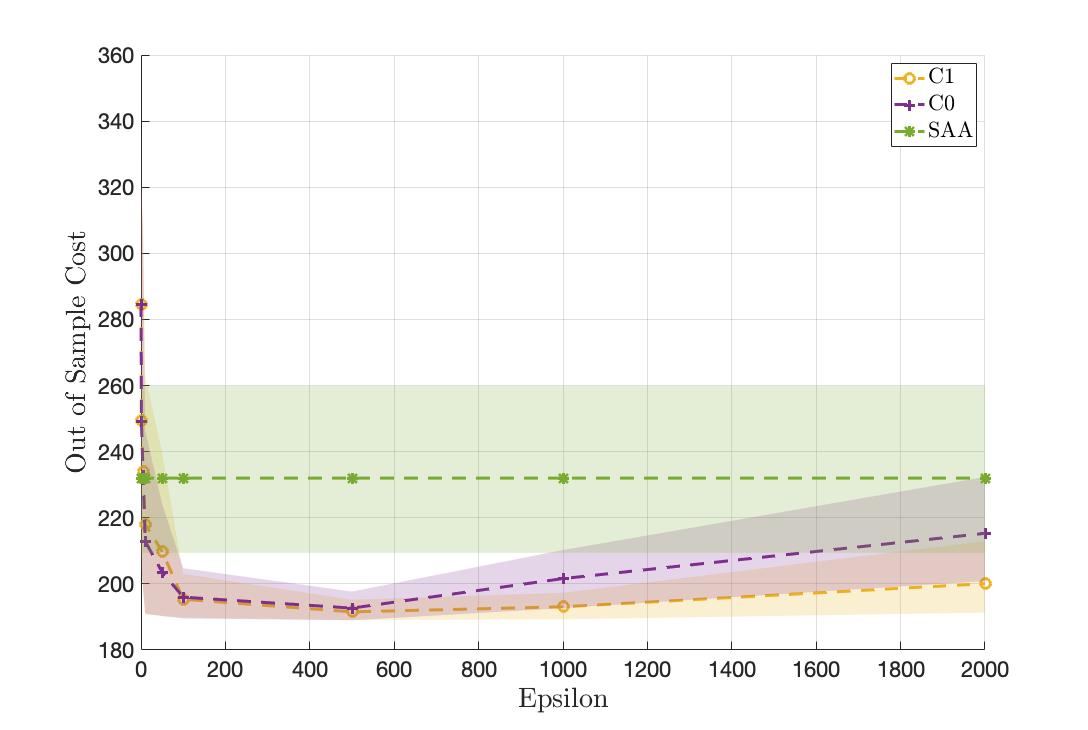}
		\end{subfigure}\hfill%
		\vspace{-3mm}
		\caption{The out-of-sample cost as a function of the number of partitions $K$ (left) with fixed in-sample data size $N = 40$, where values of epsilon are defined using cross validation. And the out-of-sample cost as a function of the value of epsilon $\epsilon_k$ (right) with fixed in-sample size and fixed number of partitions $K = N = 20$.}
		    \label{fig:newsvendor}
 		\end{center}
 	\end{figure}

When using piecewise decision rules, the number of partitions $K$ also affects the out-of-sample performance. In Figure~\ref{fig:newsvendor} (left), the out-of-sample performance improves as number of partitions increases and reaches the minimum at the point $K = N = 40$. In addition, we can see with fixed number of samples, the out-of-sample performance of \textbf{$\mathbf{C_1}$ SDP} is better than \textbf{$\mathbf{C_0}$ SDP}, which indicates $\mathcal{IA}_1\mathcal{(K)}$ provides a tighter approximation, particularly with small number of partitions $K$.
Figure~\ref{fig:newsvendor} (right) shows that varied values of epsilon provide us with different out-of-sample performances. The optimal value of epsilon under this setting is achieved in the range from 400 to 600.

\subsection{ Medical scheduling}

\subsubsection{Problem Description} 

In the medical scheduling problem, we aim to schedule the physician's daily appointments of $M$ patients in a clinic, who arrive in ascending order and may have to wait to be served. We must decide the appointment length~$x_i$ allocated to the $i^{\textup{th}}$ patient, before the realization of the unknown actual length $\xi_i$ and patient $i$'s per-unit time waiting cost in the queue $\pi_i$. The ascending arrival order means that the ${(i+1)}^{\textup{th}}$ patient's appointment is scheduled to start at time $\sum_{k=1}^{i} x_k$. The last appointment must be scheduled within the physician's regular working time~$T$, i.e., \ $\sum_{i=1}^{M} x_i \leq T$.

After the realizations of all patients' actual appointment lengths and their individual cost of waiting in the queue, we can calculate the second-stage decision variables corresponding to the waiting times of patients $y_i \geq 0$, $i \in [M]$, and the physician's total overtime $y_{M+1} \geq 0$. The following recursive formula describes the rules of waiting times $\bm y$:
$$y_{i+1} \coloneqq \max \{0, \ y_i + \xi_i - x_i\}, \ i \in [M]. $$
Here, we assume the first appointment starts at time 0, i.e., $y_1 = 0$.

Each patient's per-unit-time waiting cost $\pi_i$ is random, and the cost due to the overtime of the physician is deterministic and set to be $c$ per unit time. The goal is to schedule appointments with the minimum total waiting costs and overtime cost evaluated by worst-case CVaR:

\begin{equation*}
\begin{array}{clll}
  \displaystyle \text{min} & \displaystyle \sup_{\PP\in\mathcal P} \PP\text{-CVaR}_\epsilon[Z(\bm x, \bm \xi, \bm \pi)] \\
  \st & \bm x \in \mathbb{R}^M_{+},\ \mathbf{e}^\top \bm x \leq T, \\
\end{array}
\end{equation*}
where  
\begin{equation*}
\begin{array}{ccll}
  \displaystyle Z(\bm x, \bm \xi, \bm \pi) \coloneqq & \displaystyle \text{inf} & \displaystyle \sum_{i=1}^{M} \pi_i y_i + c y_{M+1}\\
   & \st &\bm y \in \mathbb{R}^{M+1}_{+} \\
   & & y_{i+1} \geq y_i + \xi_i - x_i \quad \forall i \in [M]. \\
\end{array}
\end{equation*}

The above optimization problem with CVaR can be reformulated as
\begin{equation}
\label{exp:MS2}
\begin{array}{clll}
   \displaystyle \text{inf} & \displaystyle \theta + \frac{1}{\delta} \sup_{\PP\in\mathcal P}\EE_\PP [\tau(\bm \xi,\bm \pi)]  \\
   \st &\theta\in \mathbb{R},\;\bm x \in \mathbb{R}^M_{+},\;\bm y: \Xi \rightarrow \RR^{M+1}, \tau: \Xi \rightarrow \RR\\
   & \mathbf{e}^\top \bm x \leq T \\
   & \left.\begin{aligned}
   & \tau(\bm \xi, \bm \pi) \geq 0, \ \tau(\bm \xi, \bm \pi) \geq \sum_{i=1}^{M} \pi_i y_i(\bm \xi, \bm \pi) + c y_{M+1}(\bm \xi, \bm \pi) - \theta \\
   & \bm y(\bm \xi, \bm \pi) \geq \bm 0, \ y_{i+1}(\bm \xi, \bm \pi) \geq y_i(\bm \xi, \bm \pi)+ \xi_i - x_i \quad \forall i \in [M] \\
\end{aligned} \right\} & \forall(\bm \xi, \bm \pi) \in \Xi.
\end{array}
\end{equation}
Here we also assume that the only information provided is the historical data of the uncertain parameters, $(\hat{\bm \xi}_1, \hat{\bm \pi}_1),\dots,(\hat{\bm \xi}_N, \hat{\bm \pi}_N)$, and the description of the support as a bounded set $ \Xi= \{ (\bm \xi, \bm \pi) \in (\RR^{M},\RR^{M}) :\linebreak \underline{\xi}_l \leq \xi_i \leq \overline{\xi}_u,\ \underline{\pi}_l \leq \pi_{i} \leq \overline{\pi}_u,\ i \in [M]\} $.

\subsubsection{Experiments}

Here we apply PQDR to the second-stage decision variable $\tau(\bm \xi, \bm \pi)$ and PLDR to the second-stage decision variables $\bm y_i(\bm \xi, \bm \pi),\ i \in [M+1]$, since we have a random recourse term $\pi_i y_i(\bm \xi, \bm \pi)$ in the constraints. After we reformulate it as a copositive program, the approximation schemes proposed in Appendix B can be applied and the above problem $\eqref{exp:MS2}$ can be reformulated as the SDP problems $\textbf{$\mathbf{C_0}$ SDP}$ and $\textbf{$\mathbf{C_1}$ SDP}$. We compare $\textbf{$\mathbf{C_0}$ SDP}$ and $\textbf{$\mathbf{C_1}$ SDP}$ with $\textbf{SAA}$ in terms of out-of-sample performance, and further apply the decomposition algorithm \textbf{Benders $\mathbf{C_0}$} to evaluate its effectiveness in reducing computation effort when solving the SDP problem.
 	
We consider the case where a physician has a schedule with $M=8$ patients and overtime cost $c=200$. To evaluate the out-of-sample performance, we calculate the average cost over 100 instances, and for each instance we generate independent training data points $(\hat{\bm \xi}_1, \hat{\bm \pi}_1),\ldots,(\hat{\bm \xi}_N, \hat{\bm \pi}_N)$ with sizes $N = 10,20,40,80,160$ and independent testing data with size 50,000. The true distributions of the actual appointment lengths $\bm \xi$ are assumed to be truncated lognormal with means $\bm \mu_1 = 4 \mathbf{e}$ and standard deviations $\bm \sigma_1 = 0.5\mathbf{e}$, and the per-unit-time waiting costs $\bm \pi$ are assumed to have the truncated lognormal distributions with means $\bm \mu_2 = \mathbf{e}$ and standard deviations $\bm \sigma_2 = 0.5\mathbf{e}$. Here we describe the support set as $\underline{\xi}_l = 20, \ \overline{\xi}_u = 100, \ \underline{\pi}_l = 1, \ \overline{\pi}_u = 10$, and set the physician’s total regular working time to $T = (\underline{\xi}_l+\overline{\xi}_u)/2 \times M = 480$. The risk attitude parameter~$\delta$ is set to be 0.1. We set the parameters $\epsilon_k, \  k \in [K]$, defined in the ambiguity set according to a 2-fold cross validation procedure and robustness parameter $\gamma$ to  0. We have the tolerance level of the decomposition algorithm set to be $\eta = 0.05$.

\subsubsection{Results Analysis}
 
  	\begin{figure}[t!]
		\includegraphics[width=0.5\linewidth]{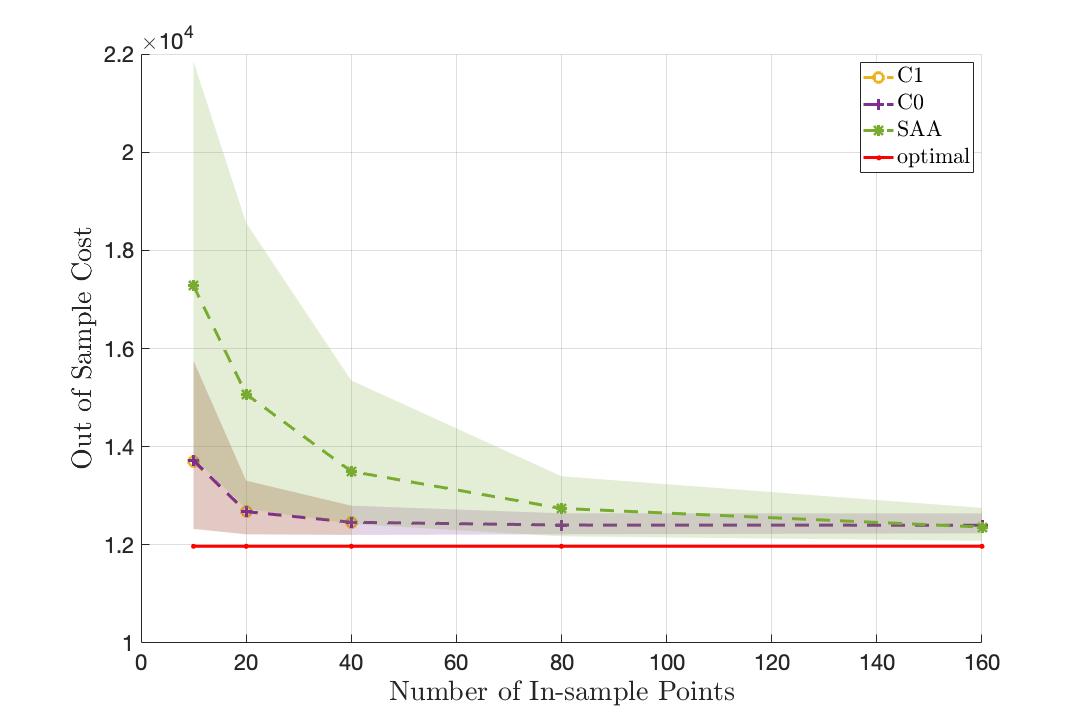}
		\caption{The out-of-sample cost as a function of the number of in-sample data points. The dashed lines represent the mean value of the out-of-sample cost, and the shaded areas visualize the $10\%$ to $90\%$ quantile range over 100 trials.}
	        \label{fig:MS_performance}
 	\end{figure}
 	
\begin{table}[t!]
\color{black}
\centering
\begin{tabular}{|c|rrrr|}\hline
$N$  & \textbf{$\mathbf{C_1}$ SDP}  & \textbf{$\mathbf{C_0}$ SDP} &  \textbf{Benders $\mathbf{C_0}$} & \textbf{SAA} \\
[0.0mm] \hline
$10$  & 9.2330 & 1.4069 & 3.7496 & 0.0028\\
$20$  & 60.5041 & 3.4428 & 3.3312 & 0.0047\\
$40$ &  866.6756 & 8.0606 & 4.2539 & 0.0069\\ 
$80$ &  $\backslash$ & 29.2646 & 6.9254 & 0.0166\\
$160$ & $\backslash$ & 152.2323 & 13.0412 & 0.0421\\
\hline\hline
\end{tabular}
\caption[]{The average runtime of different approaches over 100 trials.}
\label{tbl:MS_runtime}
\end{table}

The out-of-sample performances of \textbf{$\mathbf{C_0}$ SDP}, \textbf{$\mathbf{C_1}$ SDP} and \textbf{SAA} are shown in Figure~\ref{fig:MS_performance}. In the medical scheduling problem, \textbf{$\mathbf{C_0}$ SDP} and \textbf{$\mathbf{C_1}$ SDP} perform similarly with a small number of in-sample data. And both the SDP solution schemes generate more than $20\%$ improvements over $\textbf{SAA}$. We observe that the out-of-sample costs of the three methods decrease with the number of in-sample data points, which get closer to the optimal value as the size of in-sample data increases.

Table~\ref{tbl:MS_runtime} presents the runtimes of various methods including the one where we apply the decomposition algorithm to \textbf{$\mathbf{C_0}$ SDP}, denoted as \textbf{Benders $\mathbf{C_0}$}. The advantage of using \textbf{$\mathbf{C_0}$ SDP} rather than \textbf{$\mathbf{C_1}$ SDP} is confirmed again as \textbf{$\mathbf{C_0}$ SDP} achieves similar performance with notably less runtime. As expected, the decomposition algorithm accelerates the solving procedure of \textbf{$\mathbf{C_0}$ SDP}. 

	\begin{figure}[t!]
 		\begin{center}
		\begin{subfigure}[b]{.5\textwidth}
			\includegraphics[width=1\linewidth]{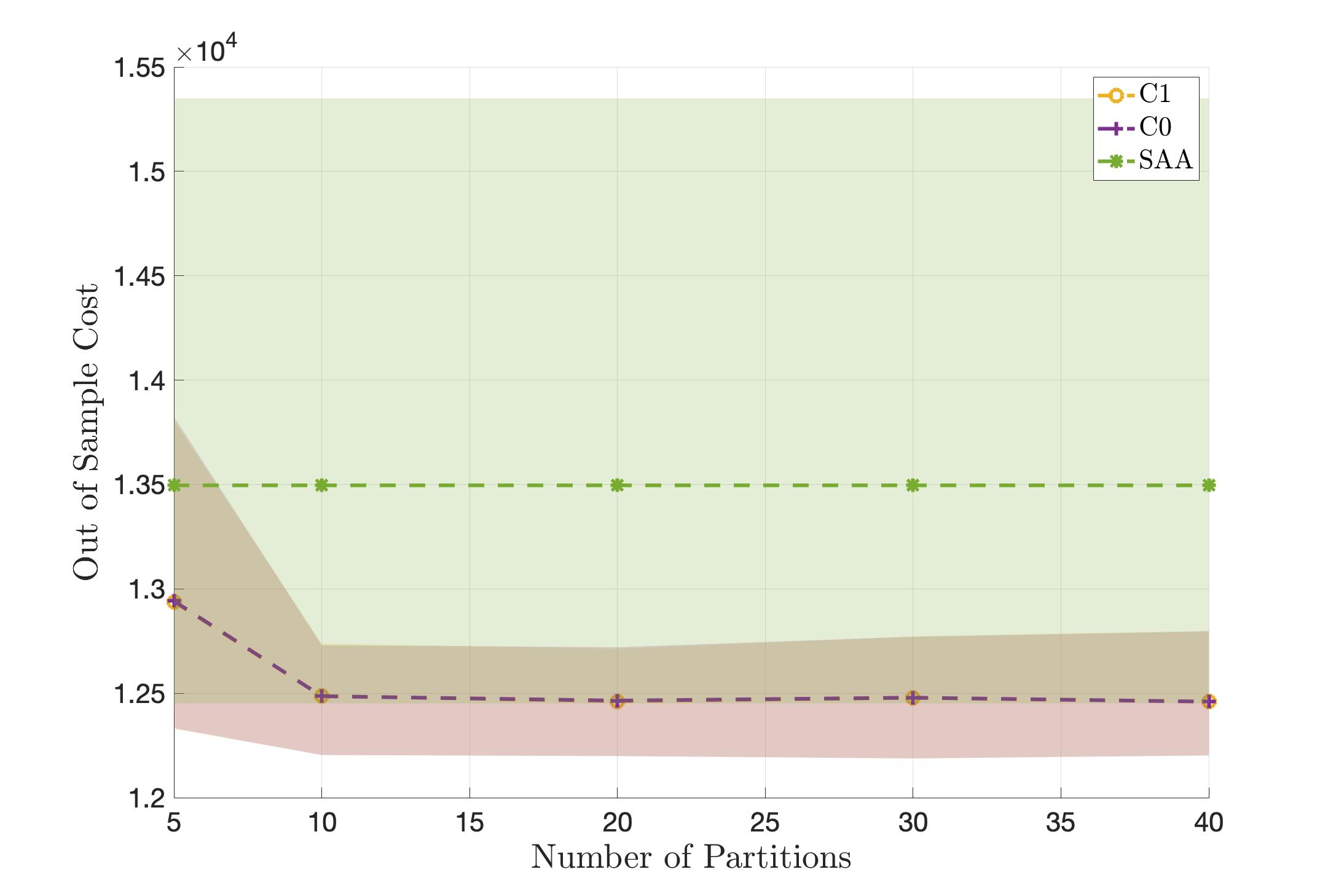}
		\end{subfigure}\hfill%
		\begin{subfigure}[b]{.5\textwidth}
			\includegraphics[width=1\linewidth]{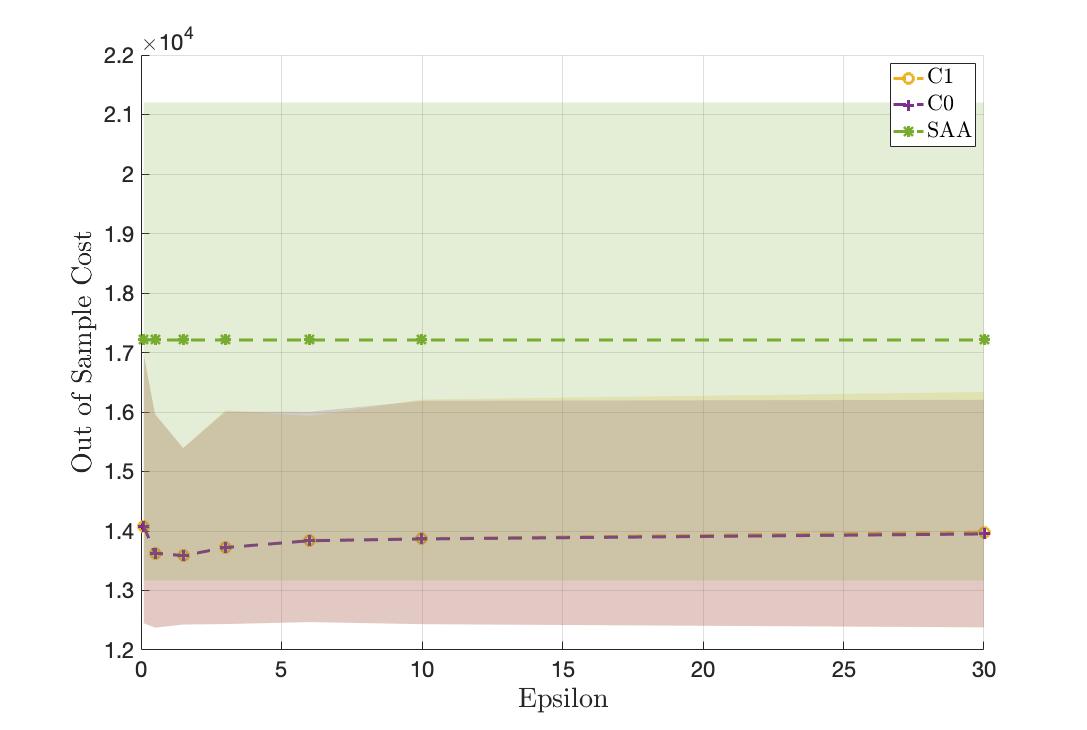}
		\end{subfigure}\hfill%
		\vspace{-3mm}
		\caption{The out-of-sample cost as a function of the number of partitions $K$ (left) with fixed in-sample data size $N = 40$, where values of epsilon are defined using cross validation. And the out-of-sample cost as a function of the value of epsilon $\epsilon_k$ (right) with fixed in-sample size and fixed number of partitions $K = N = 10$.}
		    \label{fig:MS}
 		\end{center}
 	\end{figure}
 	
We can further observe the influence of the number of partitions and the values of epsilon on the performances of \textbf{$\mathbf{C_0}$ SDP} and \textbf{$\mathbf{C_1}$ SDP} in Figures~\ref{fig:MS}. With a fixed number of in-sample samples, the larger number of partitions leads to better out-of-sample performance and the best point is reached when $K=N$. Figure~\ref{fig:MS} (right) shows that the optimal values of epsilon $\epsilon_k, \ k \in [K]$, are obtained between 0 and 3 which give the best out-of-sample performance.

\section{Conclusion}
\label{conclusion}
While the solution schemes of two-stage DRO problems with fixed recourse have been widely researched, the general settings of random recourse have so far resisted effective solution schemes due to their high intractability. 
In this paper, we leveraged the decision rule approach and derived new copositive programming reformulations for the generic two-stage DRO problems. Given the analytical generalization bounds of the proposed ambiguity set defined with two layers of robustness, we theoretically proved the finite sample guarantees on the resulting solutions. 
Furthermore, using the semidefinite approximations for the emerging copositive programs and in conjunction with the decomposition algorithm, we demonstrated that our solution method achieves near-optimal out-of-sample performance---even in view of limited sample sizes---with acceptable computation costs through a variety of numerical experiments. 
Remarkably, the decomposition algorithm empowers us to effectively solve  two-stage DRO problems with integer first-stage decisions. 

We propose two directions that deserve further investigation. First, it is interesting to explore ways to construct the partitions of the piecewise decision rule scheme other than the Voronoi regions. Different partitions may lead to  solutions with different optimality and scalability properties. Second, it is imperative to extend the proposed approach to  address multi-stage DRO problems with random recourse 
that can attain the
same scalability and attractive theoretical performance guarantee.

\section*{Acknowledgement}
We are grateful to Soroosh Shafieezadeh Abadeh for the valuable discussions and constructive suggestions which led to substantial improvements of this paper. This research was supported by the National Science Foundation grant no. 1752125.

\bibliographystyle{siam}
\bibliography{references}

\begin{thebibliography}{10}

\bibitem{LDR1}
{\sc A.~Atamtürk and M.~Zhang}, {\em Two-stage robust network flow and design
  under demand uncertainty}, Operations Research, 55 (2007), p.~662–673.

\bibitem{PDR-2}
{\sc D.~Bampou and D.~Kuhn}, {\em Scenario-free stochastic programming with
  polynomial decision rules}, in 2011 50th IEEE Conference on Decision and
  Control and European Control Conference, IEEE, 2011, pp.~7806--7812.

\bibitem{benders-dro2}
{\sc M.~Bansal, K.-L. Huang, and S.~Mehrotra}, {\em Decomposition algorithms
  for two-stage distributionally robust mixed binary programs}, SIAM Journal on
  Optimization, 28 (2018), pp.~2360--2383.

\bibitem{bayraksan2015data}
{\sc G.~Bayraksan and D.~K. Love}, {\em Data-driven stochastic programming
  using phi-divergences}, in The Operations Research Revolution, INFORMS, 2015,
  pp.~1--19.

\bibitem{PLDR-optimal}
{\sc A.~Bemporad, F.~Borrelli, and M.~Morari}, {\em Min-max control of
  constrained uncertain discrete-time linear systems}, IEEE Transactions on
  Automatic Control, 48 (2003), p.~1600–1606.

\bibitem{delta-reformulation}
{\sc A.~Ben-Tal, D.~Den~Hertog, A.~De~Waegenaere, B.~Melenberg, and G.~Rennen},
  {\em Robust solutions of optimization problems affected by uncertain
  probabilities}, Management Science, 59 (2013), pp.~341--357.

\bibitem{QDR-Slemma}
{\sc A.~Ben-Tal, L.~El~Ghaoui, and A.~Nemirovski}, {\em Robust optimization},
  Princeton university press, 2009.

\bibitem{S-lemma}
\leavevmode\vrule height 2pt depth -1.6pt width 23pt, {\em Robust
  optimization}, Princeton university press, 2009.

\bibitem{LDR2}
{\sc A.~Ben-Tal, B.~Golany, A.~Nemirovski, and J.-P. Vial}, {\em
  Retailer-supplier flexible commitments contracts: A robust optimization
  approach}, Manufacturing \& Service Operations Management, 7 (2005),
  pp.~248--271.

\bibitem{Linear-decision-rule}
{\sc A.~Ben-Tal, A.~Goryashko, E.~Guslitzer, and A.~Nemirovski}, {\em
  Adjustable robust solutions of uncertain linear programs}, Math. Program., 99
  (2004), p.~351–376.

\bibitem{PLDR1-NPhard}
{\sc A.~Ben-Tal, O.~E. Housni, and V.~Goyal}, {\em A tractable approach for
  designing piecewise affine policies in two-stage adjustable robust
  optimization}, 2018.

\bibitem{Random-Recourse-Robust-1}
{\sc A.~Ben-Tal, T.~Margalit, and A.~Nemirovski}, {\em Robust modeling of
  multi-stage portfolio problems}, in High performance optimization, Springer,
  2000, pp.~303--328.

\bibitem{PLDR4_NPhard}
{\sc D.~Bertsimas and C.~Caramanis}, {\em Finite adaptability in multistage
  linear optimization}, IEEE Transactions on Automatic Control, 55 (2010),
  pp.~2751--2766.

\bibitem{DRO-NPhard}
{\sc D.~Bertsimas, X.~V. Doan, K.~Natarajan, and C.-P. Teo}, {\em Models for
  minimax stochastic linear optimization problems with risk aversion},
  Mathematics of Operations Research, 35 (2010), pp.~580--602.

\bibitem{PLDR3-NPhard}
{\sc D.~Bertsimas and A.~Georghiou}, {\em Design of near optimal decision rules
  in multistage adaptive mixed-integer optimization}, Operations Research, 63
  (2015), p.~610–627.

\bibitem{Random-Recourse-Stochastic-1}
\leavevmode\vrule height 2pt depth -1.6pt width 23pt, {\em Binary decision
  rules for multistage adaptive mixed-integer optimization}, Mathematical
  Programming, 167 (2018), pp.~395--433.

\bibitem{DRO-solution-2}
{\sc D.~Bertsimas, V.~Gupta, and N.~Kallus}, {\em Robust sample average
  approximation}, Mathematical Programming, 171 (2018), pp.~217--282.

\bibitem{PDR-QDR1}
{\sc D.~Bertsimas, D.~A. Iancu, and P.~A. Parrilo}, {\em A hierarchy of
  near-optimal policies for multistage adaptive optimization}, IEEE
  Transactions on Automatic Control, 56 (2011), pp.~2809--2824.

\bibitem{benders-rp1}
{\sc D.~Bertsimas, E.~Litvinov, X.~A. Sun, J.~Zhao, and T.~Zheng}, {\em
  Adaptive robust optimization for the security constrained unit commitment
  problem}, IEEE transactions on power systems, 28 (2012), pp.~52--63.

\bibitem{two-stage-sampling-robust}
{\sc D.~Bertsimas, S.~Shtern, and B.~Sturt}, {\em Two-stage sample robust
  optimization}, 2020.

\bibitem{benders-sp1}
{\sc J.~BnnoBRs}, {\em Partitioning procedures for solving mixed-variables
  programming problems ‘}, Numerische mathematik, 4 (1962), pp.~238--252.

\bibitem{copositive-program}
{\sc S.~Burer}, {\em {Copositive Programming}}, in {Handbook on Semidefinite,
  Conic and Polynomial Optimization}, M.~F. Anjos and J.~B. Lasserre, eds.,
  International Series in Operations Research \& Management Science, Springer,
  March 2012, ch.~0, pp.~201--218.

\bibitem{LDR3}
{\sc G.~C. Calafiore}, {\em Multi-period portfolio optimization with linear
  control policies}, Automatica, 44 (2008), p.~2463–2473.

\bibitem{LDR4}
{\sc X.~Chen, M.~Sim, and P.~Sun}, {\em A robust optimization perspective on
  stochastic programming}, Operations Research, 55 (2007), p.~1058–1071.

\bibitem{LDR5}
{\sc X.~Chen, M.~Sim, P.~Sun, and J.~Zhang}, {\em A linear decision-based
  approximation approach to stochastic programming}, Operations Research, 56
  (2008), pp.~344--357.

\bibitem{DRO-SLDR2}
\leavevmode\vrule height 2pt depth -1.6pt width 23pt, {\em A linear
  decision-based approximation approach to stochastic programming}, Operations
  Research, 56 (2008), pp.~344--357.

\bibitem{DRO-solution-3}
{\sc Z.~Chen, M.~Sim, and P.~Xiong}, {\em Robust stochastic optimization made
  easy with rsome}, Management Science, 66 (2020), pp.~3329--3339.

\bibitem{Random-Recourse-Stochastic-2}
{\sc G.~B. Dantzig and G.~Infanger}, {\em Multi-stage stochastic linear
  programs for portfolio optimization}, Annals of Operations Research, 45
  (1993), pp.~59--76.

\bibitem{DRO-tractable-1}
{\sc E.~H. Delage}, {\em Distributionally robust optimization in context of
  data-driven problems}, Stanford University, 2009.

\bibitem{esfahani2018data}
{\sc P.~M. Esfahani and D.~Kuhn}, {\em Data-driven distributionally robust
  optimization using the wasserstein metric: Performance guarantees and
  tractable reformulations}, Mathematical Programming, 171 (2018),
  pp.~115--166.

\bibitem{gao2020finite}
{\sc R.~Gao}, {\em Finite-sample guarantees for wasserstein distributionally
  robust optimization: Breaking the curse of dimensionality}, arXiv preprint
  arXiv:2009.04382,  (2020).

\bibitem{GBD1}
{\sc A.~M. Geoffrion}, {\em Generalized benders decomposition}, Journal of
  optimization theory and applications, 10 (1972), pp.~237--260.

\bibitem{PLDR2}
{\sc A.~Georghiou, W.~Wiesemann, and D.~Kuhn}, {\em Generalized decision rule
  approximations for stochastic programming via liftings}, Mathematical
  Programming, 152 (2014), p.~301–338.

\bibitem{DRO-SLDR1}
{\sc J.~Goh and M.~Sim}, {\em Distributionally robust optimization and its
  tractable approximations}, Operations research, 58 (2010), pp.~902--917.

\bibitem{LDR6}
{\sc C.~E. Gounaris, W.~Wiesemann, and C.~A. Floudas}, {\em The robust
  capacitated vehicle routing problem under demand uncertainty}, Operations
  Research, 61 (2013), p.~677–693.

\bibitem{Random-Recourse-NPhard}
{\sc E.~Guslitser}, {\em Uncertainty-immunized solutions in linear
  programming}, PhD thesis, Citeseer, 2002.

\bibitem{Wass-SDP}
{\sc G.~A. Hanasusanto and D.~Kuhn}, {\em Conic programming reformulations of
  two-stage distributionally robust linear programs over wasserstein balls},
  Operations Research, 66 (2018), p.~849–869.

\bibitem{QDR-2}
{\sc G.~A. Hanasusanto, D.~Kuhn, S.~W. Wallace, and S.~Zymler}, {\em
  Distributionally robust multi-item newsvendor problems with multimodal demand
  distributions}, Mathematical Programming, 152 (2015), pp.~1--32.

\bibitem{LDR7}
{\sc Z.~Hao, L.~He, Z.~Hu, and J.~Jiang}, {\em Robust vehicle pre-allocation
  with uncertain covariates}, Production and Operations Management, 29 (2020),
  pp.~955--972.

\bibitem{benders-rp2}
{\sc H.~Hashemi~Doulabi, P.~Jaillet, G.~Pesant, and L.-M. Rousseau}, {\em
  Exploiting the structure of two-stage robust optimization models with
  exponential scenarios}, INFORMS Journal on Computing, 33 (2021),
  pp.~143--162.

\bibitem{DRO-solution-1}
{\sc R.~Jiang and Y.~Guan}, {\em Risk-averse two-stage stochastic program with
  distributional ambiguity}, Operations Research, 66 (2018), pp.~1390--1405.

\bibitem{Random-Recourse-Stochastic}
{\sc D.~Kuhn, W.~Wiesemann, and A.~Georghiou}, {\em Primal and dual linear
  decision rules in stochastic and robust optimization}, Mathematical
  Programming, 130 (2011), pp.~177--209.

\bibitem{delta-convergence}
{\sc B.~Laurent and P.~Massart}, {\em Adaptive estimation of a quadratic
  functional by model selection}, Annals of Statistics,  (2000),
  pp.~1302--1338.

\bibitem{Random-Recourse-Stochastic-3}
{\sc P.~C. Martins~da Silva~Rocha}, {\em Medium-term planning in deregulated
  energy markets with decision rules},  (2013).

\bibitem{minimum-ellipsoids}
{\sc A.~Mittal and G.~A. Hanasusanto}, {\em Finding minimum volume
  circumscribing ellipsoids using generalized copositive programming}, arXiv
  preprint arXiv:1807.07507,  (2018).

\bibitem{phi-divergence}
{\sc L.~Pardo}, {\em Statistical inference based on divergence measures}, CRC
  press, 2018.

\bibitem{LDR8}
{\sc K.~Postek and D.~d. Hertog}, {\em Multistage adjustable robust
  mixed-integer optimization via iterative splitting of the uncertainty set},
  INFORMS Journal on Computing, 28 (2016), p.~553–574.

\bibitem{Random-Recourse-Stochastic-4}
{\sc P.~Rocha and D.~Kuhn}, {\em Multistage stochastic portfolio optimisation
  in deregulated electricity markets using linear decision rules}, European
  Journal of Operational Research, 216 (2012), pp.~397--408.

\bibitem{CVaR-def}
{\sc R.~T. Rockafellar, S.~Uryasev, et~al.}, {\em Optimization of conditional
  value-at-risk}, Journal of risk, 2 (2000), pp.~21--42.

\bibitem{minimax-thm}
{\sc A.~Shapiro and A.~Kleywegt}, {\em Minimax analysis of stochastic
  problems}, Optimization Methods and Software, 17 (2002), pp.~523--542.

\bibitem{second-moment-bound}
{\sc J.~Shawe-Taylor and N.~Cristianini}, {\em Estimating the moments of a
  random vector with applications}, in Proceedings of GRETSI 2003 Conference,
  2003, pp.~47--52.
\newblock Invited Talk.

\bibitem{benders-sp2}
{\sc R.~M. Van~Slyke and R.~Wets}, {\em L-shaped linear programs with
  applications to optimal control and stochastic programming}, SIAM Journal on
  Applied Mathematics, 17 (1969), pp.~638--663.

\bibitem{MSRO-decision-rules}
{\sc G.~Xu and G.~A. Hanasusanto}, {\em Improved decision rule approximations
  for multi-stage robust optimization via copositive programming}, 2018.

\bibitem{benders-dro1}
{\sc C.~Zhao and Y.~Guan}, {\em Data-driven risk-averse stochastic optimization
  with wasserstein metric}, Operations Research Letters, 46 (2018),
  pp.~262--267.

\end{thebibliography}

\newpage

\appendix

\section{Proofs}
\subsection{Proof of Proposition 1}
\begin{proof}
Recall that the definition of CVaR at level $\delta$ regarding to the distribution $\PP$ is given by
$$\PP\text{-CVaR}_\delta [Z(\bm x, \bm \xi)] = \inf_{\theta \in \RR} \left \{ \theta + \frac{1}{\delta} \EE_{\PP} \left ( [Z(\bm x, \bm \xi) - \theta] ^{+}  \right ) \right \},$$
see \cite{CVaR-def}. According to a stochastic min-max theorem by Shapiro and Kleywegt \cite{minimax-thm}, the worst-case CVaR in the objective function can be reformulated as
$$\sup_{\PP\in \mathcal P} \PP\text{-CVaR}_\delta [Z(\bm x, \bm \xi)] = \inf_{\theta \in \RR}  \theta + \frac{1}{\delta} \sup_{\PP\in \mathcal P} \EE_{\PP} \left ( [Z(\bm x, \bm \xi) - \theta] ^{+} \right ).$$
By defining the epigraphical variable $\tau(\bm \xi)$ to represent $\max\left \{ (\bm D \bm \xi)^\top \bm y (\bm \xi) - \theta, 0 \right \}$ in the objective, 
and introducing two  inequalities $\tau(\bm \xi) \geq 0$ and $ \tau(\bm \xi) \geq (\bm D \bm \xi)^\top \bm y (\bm \xi) - \theta$ in the constraints, we obtain the semi-infinite linear program~\eqref{eq:2s-dro2}.
Thus, the claim follows. \end{proof}

\subsection{Proof of Lemma 6}
\begin{proof}
After applying linear decision rules, checking complete recourse is equivalent to checking the feasibility of the constraint
\begin{equation}
\label{complete_recourse_LDR1}
\begin{array}{clll}
 0 < & \displaystyle \inf_{\bm \xi \in \RR} & \bm \xi^\top \bm W_\ell^\top \bm Y \bm\xi \\
 &\st & \mathbf{e}_{S+1}^\top \bm \xi = 1\\
 & & \bm\xi \in  \mathcal K, \\
 \end{array}
\end{equation}
for each $\ell \in [L]$. 
From Lemma \ref{lem2}, \eqref{complete_recourse_LDR1} is equivalent to the constraint involving the CPP whose dual problem leads to a  new constraint given by  
\begin{equation}
\label{complete_recourse_LDR2}
\begin{array}{cll}
 0 < & \displaystyle \sup & \beta_\ell \\
 &\st & \beta_\ell \in \RR, \ \bm Y \in \RR^{N_2 \times (S+1)}\\
 & & \frac{1}{2} \left (\bm Y^\top \bm W_\ell + \bm W_\ell^\top \bm Y \right ) - \beta_\ell \mathbf{e}_{S+1} \mathbf{e}_{S+1}^\top 
  \in \mathcal{C} (\mathcal K). \\
\end{array}
\end{equation}
Strong duality holds based on Lemma \ref{lem3}. Furthermore, the proof of the lemma implies that the optimal value of the dual problem in \eqref{complete_recourse_LDR2} is attained. Thus, the constraint is satisfied if and only if there exists $\beta_\ell \in \RR_{++}$ such that 
\begin{equation}
    \frac{1}{2}\left (\bm Y^\top \bm W_\ell + \bm W_\ell^\top \bm Y \right ) - \beta_\ell \mathbf{e}_{S+1} \mathbf{e}_{S+1}^\top \in {\mathcal{C}(\mathcal{K})} 
\end{equation}
is satisfied. This completes the proof. 
\end{proof}

\subsection{Proof of Proposition 5}
\begin{proof}
Satisfying complete recourse under linear decision rules, we conclude that there exists $\bm Y_k \in \RR^{N_2 \times (S+1)}$ such that $(\bm W_\ell\bm \xi)^\top \bm Y_k \bm \xi > 0$ for all $\bm \xi \in \Xi$ and $\ell \in [L]$.

As illustrated in Section \ref{copositive_reformulations}, the constraints in $\eqref{eq:pdr_ref2_z}$
\begin{equation*}
\left \{
\begin{aligned}
 & \pi^k_\ell + \kappa_\ell \theta \geq 0 \quad \forall \ell \in [L+2] \\
 & \bm \Delta^k_\ell(\bm x,\bm Y_k,\bm Q_k) - \pi^k_\ell \mathbf{e}_{S+1} \mathbf{e}_{S+1}^\top 
\in \mathcal{C} (\mathcal K_k) \quad \forall \ell \in [L+2]  \\
\end{aligned}
\right.
\end{equation*}
are equivalent to the following constraints:
\begin{equation}
\label{eq:pdr_cons1_6}
\bm{\mathcal{T}}_\ell (\bm x)^\top\bm \xi  \leq (\bm{\mathcal{W}}_\ell \bm\xi)^\top\bm Y_k\bm\xi + \lambda_\ell \bm \xi^\top \bm Q_k \bm \xi + \kappa_\ell \theta  \quad \forall \bm\xi \in \Xi_k \ \forall \ell \in [L+2].    
\end{equation}
For $\ell \in [L]$, there is always an $a_k^\ell > 0$ such that $a_k^\ell(\bm W_\ell \bm \xi)^\top \bm Y_k \bm \xi$ exceeds $\bm T_\ell(\bm x)^\top \bm \xi$ for all $\bm \xi$ in the bounded support set $\Xi_k$. By setting $\bm Y_k\leftarrow (\max_{\ell\in[L]}a^\ell_k) \bm Y_k$, the strict inequality of \eqref{eq:pdr_cons1_6} holds. For $\ell \in \{L+1, L+2\}$, there always exists $\bm Q_k \in \RR^{(S+1) \times (S+1)}$ such that $0 < \bm \xi^\top \bm Q_k \bm \xi$ and $(\bm D \bm \xi)^\top \bm Y_k \bm \xi - \theta< \bm \xi^\top \bm Q_k \bm \xi$ for all $\bm \xi$ in the bounded set $\Xi_k$ because $\bm Q_k$ is a free variable that does not appear in other constraints in $\eqref{eq:pdr_cons1_6}$. Thus, the strict inequality of \eqref{eq:pdr_cons1_6} holds.

Analogously, since $\bm B_k$ and $\alpha_k$ are free variables that do not appear in other constraints in $\eqref{eq:pdr_ref2_z}$, with nonempty cone $\mathcal{K}_k$, there always exist $\bm B_k \in \mathbb{S}^{S+1}, \ \alpha_k \in \RR$ such that
\begin{equation*}
    \bm B_k + \alpha_k \mathbf{e}_{S+1} \mathbf{e}_{S+1}^\top \succ_{\mathcal{C}(\mathcal{K}_k)} \bm 0.
\end{equation*}

Based on the above illustration, we conclude that the primal problem $\eqref{eq:pdr_ref2_z}$ satisfies the Slater's condition. Thus, the claim follows. 
\end{proof}



\section{Semidefinite Programming Solution Schemes}

We consider the generic structure of the cone $\mathcal{K}_k$ under Assumption 1 as the intersection of polyhedral and second-order cones
\begin{equation*}
\begin{array}{clll}
\mathcal K_k:= \{ \bm \xi \in \RR^{S} \times \RR_{+} :\hat{\bm P}_k \bm \xi \geq \bm 0, \ \hat{\bm R}_k \bm \xi\in \mathcal{SOC}(K_r) \},
\end{array}
\end{equation*}
where $\hat{\bm P}_k \in \RR^{S_p \times (S+1)}$ and $\hat{\bm R}_k \in \RR^{S_r \times (S+1)}$. The corresponding support sets $\Xi_k, \ k \in [K]$, are polyhedral and second-order cone representable.

The recently proposed semidefinite approximation schemes to the cone $\mathcal{C} (\mathcal K_k)$ are described in the following propositions. 
\begin{proposition} \textup{(\cite[Theorem B.3.1]{S-lemma}).}
A conservative semidefinite-representable approximation scheme to the generalized copositive cone $\mathcal{C} (\mathcal K_k)$ in view of the approximate S-lemma is given by
\begin{equation}
\label{eq:sdp1}
\begin{array}{clll}
     \mathcal{IA}_0(\mathcal K_k) \coloneqq \left\{ \bm V_k \in \mathbb{S}^{S+1}: \begin{matrix} \tau_k \geq 0, \ \bm \beta_k \in \RR^{S_p}_{+}, \ \bm U_k \in \mathbb{S}^{S+1}, \ \bm U_k \succeq \bm 0\\
     \bm V_k = \bm U_k + \tau \hat{\bm M}_k + \frac{1}{2}(\hat{\bm P}_k^\top \bm \beta_k \mathbf{e}^\top_{S+1} + \mathbf{e}_{S+1}\bm \beta_k^\top\hat{\bm P}_k)
    \end{matrix}\right\},
\end{array}
\end{equation}
where the matrix $\hat{\bm M}_k \in \mathbb{S}^{S+1}$ is defined as
\begin{equation}
\label{eq:sdp_m}
\begin{array}{clll}
\hat{\bm M}_k:= \hat{\bm R}_k^\top \mathbf{e}_{S_r}\mathbf{e}_{S_r}^\top \hat{\bm R}_k - \sum_{\ell =  1}^{S_r - 1}\hat{\bm R}_k^\top \mathbf{e}_{\ell}\mathbf{e}_{\ell}^\top \hat{\bm R}_k.
\end{array}
\end{equation}
\end{proposition}

\begin{proposition}\textup{(\cite[Proposition 1]{MSRO-decision-rules}).}
Another inner semidefinite approximation scheme to $\mathcal{C} (\mathcal K_k)$ is given by
\begin{equation}
\label{eq:sdp2}
\begin{array}{clll}
     \mathcal{IA}_1 (\mathcal K_k) \coloneqq \left\{ \bm V_k \in \mathbb{S}^{S+1}: \begin{matrix}\bm U_k \in \mathbb{S}^{S+1}, \ \bm U_k \succeq \bm 0, \ \bm \Sigma_k \in \mathbb{S}^{S_p}\\
     \bm \Psi_k \in \mathbb{S}^{S+1}, \ \bm \Phi_k \in \RR^{S_p \times S_r}, \ \tau_k \in \RR_{+}\\
     \bm V_k = \bm U_k + \tau_k \hat{\bm M}_k + \hat{\bm P}_k^\top \bm \Sigma_k \hat{\bm P}_k + \bm \Psi_k, \bm \Sigma_k \geq \bm 0\\
     \bm \Psi_k = \frac{1}{2}(\hat{\bm P}_k^\top \bm \Phi_k \hat{\bm R}_k  + \hat{\bm R}_k^\top \bm \Phi_k^\top \hat{\bm P}_k), \ \textup{Rows}(\bm \Phi_k) \in \mathcal{SOC}(S_r)
    \end{matrix}\right\},
\end{array}
\end{equation}
where $\hat{\bm M}_k$ is defined in \eqref{eq:sdp_m}.
\end{proposition}

Replacing the copositive cone $\mathcal C(\mathcal K_k)$ in the constraints of $\eqref{eq:pdr_ref}$ with the inner approximations $\mathcal IA_0 (\mathcal K_k)$ and $\mathcal IA_1 (\mathcal K_k), \ k \in [K]$, separately gives rise to conservative semidefinite programs, whose optimal objective values are denoted as $J^{\textup{PDR}}_0$ and $J^{\textup{PDR}}_1$, respectively. The following theorem compares the performances of these semidefinite approximations.
\begin{theorem} \textup{(\cite[Proposition 1 and 3]{MSRO-decision-rules}).} \label{thm4} We have 
\begin{equation*}
    \mathcal{IA}_0 (\mathcal K_k) \subseteq \mathcal{IA}_1 (\mathcal K_k) \subseteq \mathcal{C}(\mathcal K_k) \quad \forall k \in [K].
\end{equation*} 
Thus, the following chains of inequalities hold: 
\begin{equation*}
    J^{\textup{PDR}} \leq J^{\textup{PDR}}_1 \leq J^{\textup{PDR}}_0.
\end{equation*}  
\end{theorem}
Theorem \ref{thm4} establishes that both $\mathcal{IA}_0 (\mathcal K_k)$ and $\mathcal{IA}_1 (\mathcal K_k)$ are subsets of $\mathcal{C}(\mathcal K_k)$. Additionally, $\mathcal{IA}_0 (\mathcal K_k)$ is inferior to $\mathcal{IA}_1 (\mathcal K_k)$ with regard to the performance of approximating $\mathcal{C}(\mathcal K_k), \ k \in [K]$. 
Nonetheless, it is essential to note that the simpler semidefinite program with approximation $\mathcal{IA}_0 (\mathcal K_k)$ requires significantly less runtime. 
We show in the experiments that $\mathcal{IA}_0 (\mathcal K_k)$ is preferable when balancing the suboptimality and scalability of the inner approximations. 

Next, we develop the semidefinite approximations to the completely positive cone  $\mathcal{C^*(K)}$, which are the duals of the primal approximations $\eqref{eq:sdp1}$ and $\eqref{eq:sdp2}$.

\begin{proposition}
The outer semidefinite-representable approximation to the completely positive cone $\mathcal{C^*}(\mathcal{K}_k)$ corresponding to $\mathcal{IA}_0 (\mathcal K_k)$ is given by
\begin{equation}
\label{eq:sdp1-dual}
\begin{array}{cl}
     \mathcal{OA}_0 (\mathcal K_k) \coloneqq \left\{ \bm F_k \in \mathbb{S}^{S+1}: \begin{matrix} \bm F_k \succeq \bm 0\\
     \displaystyle \tr \left( \bm F_k \hat{\bm M}_k \right) \geq 0\\
     \hat{\bm P}_k \bm F_k \mathbf{e}_{S+1} \geq \bm 0
    \end{matrix}\right\},
\end{array}
\end{equation}
where $\hat{\bm M}_k$ is defined as in \eqref{eq:sdp_m}.
\end{proposition}

\begin{proposition}
The semidefinite approximation scheme to the cone $\mathcal{C^*}(\mathcal{K}_k)$ in accordance with $\mathcal{IA}_1 (\mathcal K_k)$ is given by
\begin{equation}
\label{eq:sdp2-dual}
\begin{array}{cl}
     \mathcal{OA}_1 (\mathcal K_k) \coloneqq \left\{ \bm F_k \in \mathbb{S}^{S+1}: \begin{matrix} \bm F_k \succeq \bm 0\\
     \displaystyle \tr \left( \bm F_k \hat{\bm M}_k \right) \geq 0\\
     \hat{\bm P}_k \bm F_k \hat{\bm P}_k^\top \geq \bm 0 \\
     \textup{Rows}(\hat{\bm P}_k \bm F_k \hat{\bm R}_k^\top) \in \mathcal{SOC}(S_r)
    \end{matrix}\right\},
\end{array}
\end{equation}
where $\hat{\bm M}_k$ is defined as in \eqref{eq:sdp_m}.
\end{proposition}

Replacing the completely positive cones $\mathcal C^*(\mathcal K_k), \ k \in [K]$, in the constraints of the problems $\eqref{eq:benders_sp}$ and~$\eqref{eq:benders_fea_cut}$ with the outer approximations $\mathcal OA_0 (\mathcal K_k)$ and $\mathcal OA_1 (\mathcal K_k), \ k \in [K]$, respectively, provides us with tractable semidefinite programs in the decomposition algorithm. It is evident that the outer semidefinite approximation $\mathcal{OA}_0 (\mathcal K_k)$ leads to a simpler semidefinite program with less computation cost. Furthermore, $\mathcal{OA}_0 (\mathcal K_k)$ can be shown as a better approximation scheme balancing the out-of-sample performance and computational effort, which is consistent with the inner approximation $\mathcal{IA}_0 (\mathcal K_k)$.

\section{Facility Location Problem}

\subsection{Problem Description} 
We study a facility location problem with $I$ facilities and $J$ customer locations. We want first to select the subset of $I$ facilities to open before knowing the customers' demands $d_j$, $j \in [J]$. After the realizations of the demands $\bm d$, we decide the amount of products to serve customers with the opened facilities, in order to minimize the total fixed costs of allocating facilities and the expected transportation costs to serve customers. Note that in practice, it is possible that the total demands exceed the total supplies. In this circumstance, out-of-stock loss is incurred. We have the binary first-stage decision variables, i.e., $\bm x \in \{ 0,1\}^I$, where $x_i = 1$ indicates opening facility $i$ while $x_i = 0$ means not using facility $i$. The second-stage decision variable~$y_{ij}$ determines the fraction of customer location $j$'s demand fulfilled by the $i^{\textup{th}}$ facility, and $w_{j}$ represents the fraction of customer location $j$'s demand unsatisfied by opened facilities as each facility has its production capacity $u_i$. The two-stage optimization model is given by
\begin{equation}
\label{exp:facility-location-problem}
\begin{array}{clll}
  \displaystyle  \text{min} & \displaystyle \bm f^\top \bm x + \sup_{\PP\in\mathcal P}\EE_\PP [Z (\bm x, \bm d)] \\
  \st & \bm x \in \{0,1\}^{I}, \\ 
\end{array}
\end{equation}
where 
\begin{equation*}
\begin{array}{ccll}
  \displaystyle Z(\bm x, \bm d) \coloneqq & \displaystyle \text{min} & \displaystyle \sum_{i \in [I]} \sum_{j \in [J]} c_{ij} d_{j} y_{ij} + \sum_{j \in [J]} g_{j} d_{j} w_{j} \\
  & \st & \bm y \in \mathbb{R}_{+}^{I \times J}, \ \bm w \in \mathbb{R}_{+}^{J} \\
  &  & \underset{i \in [I]}{\sum} y_{ij} + w_j \geq 1 \quad  \forall j \in [J] \\
  &  & \underset{j \in [J]}{\sum} d_j y_{ij} \leq u_i x_i \quad  \forall i \in [I]. \\ 
\end{array}
\end{equation*}

In the objective function, the fixed cost of opening facility $i$ is denoted as $f_i$. The transportation cost~$c_{ij}$ represents the per-unit cost of using facility $i$ to supply customer location $j$. The unit out-of-stock loss of customer location $j$, denoted by $g_j$, can incorporate either the external supply cost or the customer dissatisfaction due to unsatisfied demand. It should be sufficiently large so that the stock-out cost is invoked only when the total supply from all facilities cannot meet the demand. The first set of constraints ensures that the demands of all customers are satisfied. The second set of constraints requires that the service provided by each facility does not exceed its production capacity, and only selected facilities can service customers.

Similarly to previous experiments, we assume that the only information provided is the historical data of the uncertain demands, $\hat{\bm d}_1,\ldots,\hat{\bm d}_N$, and the description of the support set as a bounded set $ \Xi= \{ \bm d \in \RR^{J}: \linebreak \underline{d}_l \leq d_j \leq \overline{d}_u, \ j \in [J]\} $.

\subsection{Experiments}

The optimization problem \eqref{exp:facility-location-problem} has binary first-stage decision variable $\bm x$. 
After applying PLDR to the second-stage decision variables $\bm y$ and $\bm w$, the problem can be reformulated as a mixed-binary COP. The mixed-binary conic programming problem leads to high computation cost even if it is approximated as a mixed-binary semidefinite program. Thus, we apply the decomposition algorithm, denoted as $\textbf{Benders $\mathbf{C_0}$}$, where the master problem is a mixed-binary SOCP and the subproblems are semidefinite programs. We compare the out-of-sample performance and runtime of $\textbf{Benders $\mathbf{C_0}$}$ with $\textbf{SAA}$.

We consider a facility location problem with $I=5$ available facilities and $J=5$ customer locations. To evaluate the qualities of solutions obtained from $\textbf{Benders $\mathbf{C_0}$}$ and $\textbf{SAA}$, we randomly generate i.i.d. training samples of uncertain demands $\{\hat{\bm d}_i\}_{i \in [N]}$ with sizes $N = 10,20,40,80,160$, and testing data with size 50,000 from the truncated lognormal distribution with means $\bm \mu = 6\mathbf{e}$ and standard deviations $\bm \sigma = \mathbf{e}$. The out-of-sample costs and runtimes are calculated by averaging over 100 trials. The uncertain demands range from the lower bound $\underline{d}_l = 200$ to the upper bound $\overline{d}_u = 3000$. We set the per-unit out-of-stock loss to $\bm g = 1000\mathbf{e}$, while the unit transportation cost, $c_{ij}$, is uniformly chosen in $[10,100]$ and the fixed cost of opening facility $i$, $f_i$, is uniformly chosen in $[4000,5000]$. We have the same capacity of all facilities, i.e., $\bm u = 1500\mathbf{e}$. The robust parameters $\epsilon_k, \ k \in [K]$, describing the ambiguity set are set according to the 2-fold cross validation, and the parameter $\gamma$ is set to be 0. The tolerance level of the decomposition algorithm $\eta$ has the value 0.05.

\subsection{Results Analysis}

\floatsetup{heightadjust=all, floatrowsep=columnsep}
\newfloatcommand{figurebox}{figure}
\newfloatcommand{tablebox}{table}

\begin{figure}
\begin{floatrow}[2]
\figurebox{\caption{The out-of-sample cost as a function of in-sample data size.}\label{fig:FLP_performance}}{\includegraphics[width=7.5cm]{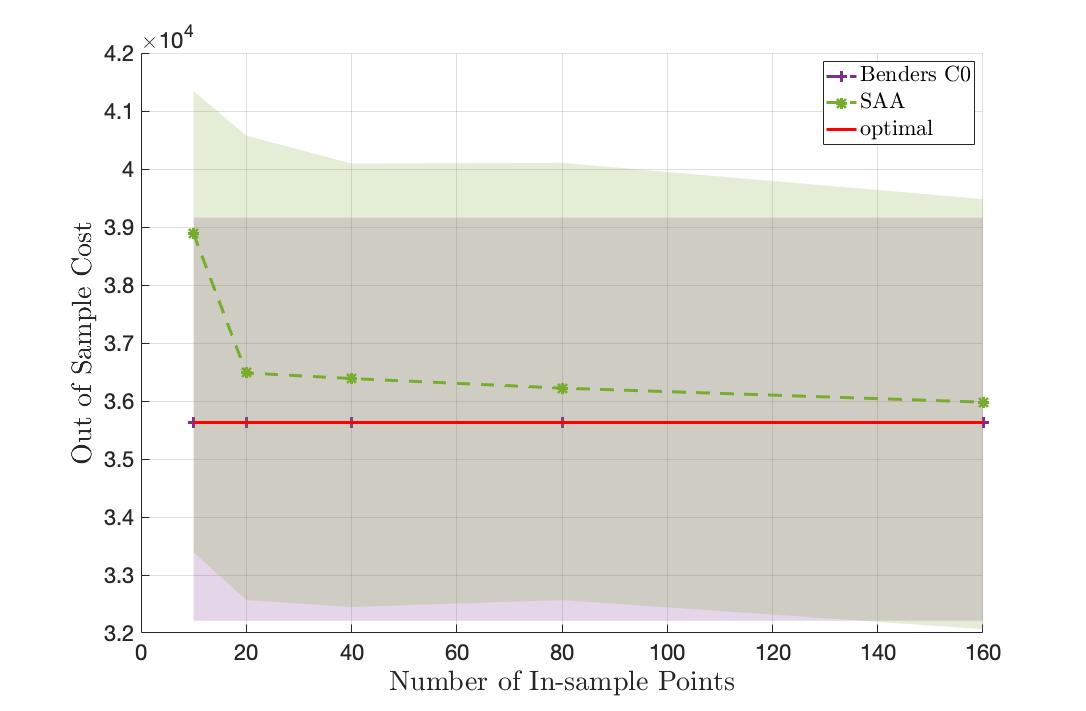}}
\tablebox{\caption{The average runtime of calculated over 100 trials.}\label{tbl:FLP_runtime}}{
      \begin{tabular}{|c|rr|}\hline
    $N$  &  \textbf{Benders $\mathbf{C_0}$} & \textbf{SAA} \\
    [0.0mm] \hline
    $10$  & 0.7684 & 0.0232 \\
    $20$  & 0.8255 & 0.0321 \\
    $40$ &  0.9542 & 0.0587 \\ 
    $80$ &  1.7544 & 0.1034 \\
    $160$ & 4.1704 & 0.1215 \\
\hline\hline
\end{tabular}}
\end{floatrow}
\end{figure}

Figure~\ref{fig:FLP_performance} shows the out-of-sample performances of the two proposed solution methods. It is noticeable that $\textbf{Benders $\mathbf{C_0}$}$ provides the optimal solution even with only $N = 10$ historical data. Meanwhile, $\textbf{SAA}$ has decreasing out-of-sample cost, approaching the optimal value as $N$ increases. When the in-sample data size reaches $N = 160$, the performances of these two methods are getting fairly close. In Table~\ref{tbl:FLP_runtime}, $\textbf{SAA}$ is highly computationally efficient as expected, in the same time, $\textbf{Benders $\mathbf{C_0}$}$ achieves better performance and requires only a few seconds to solve which is attractive. 

\end{document}